\documentclass{amsart}[11pt]
\usepackage{amsmath}
\usepackage{amssymb}
\usepackage{mathrsfs}
\usepackage{amsfonts}
\usepackage{graphicx}
\usepackage{tikz}
\usepackage{texdraw}
\usepackage{graphpap}
\usepackage{enumitem}
\usepackage{chngcntr}
\usepackage{wasysym}
\usepackage{color}
\usepackage{tikz}
\usepackage{todonotes}
\usepackage{comment}
\usepackage{soul}

\usetikzlibrary{decorations.markings}
\usetikzlibrary{spy}

\newcommand{\boundellipse}[3]
{(#1) ellipse (#2 and #3)
}

\theoremstyle{plain}
\newtheorem{thm}{Theorem}[section]

\newtheorem{cor}[thm]{Corollary}
\newtheorem{lem}[thm]{Lemma}

\newtheorem{prop}[thm]{Proposition}

\theoremstyle{remark}
\newtheorem*{defn}{Definition}
\newtheorem*{rem}{Remark}

\numberwithin{equation}{section}

\newcommand{\He}{\mathrm{He}}

\newcommand{\fluct}{\operatorname{fluct}}

\newcommand{\Peak}{P}
\newcommand{\SignPeak}{SLP}
\newcommand{\LocPeak}{LP}

\newcommand{\calB}{{\mathcal B}}
\newcommand{\calN}{{\mathcal N}}

\newcommand{\calE}{{\mathcal E}}

\newcommand{\Prob}{{\mathbb{P}}}

\newcommand{\Osc}{{\mathcal G}}

\newcommand{\R}{{\mathbb R}}

\newcommand{\Z}{{\mathbb Z}}
\newcommand{\D}{{\mathbb D}}

\newcommand{\C}{{\mathbb C}}

\newcommand{\E}{{\mathbb E}}

\newcommand{\eps}{{\varepsilon}}

\newcommand{\Qloc}{Q_{\mathrm{loc}}}

\renewcommand{\d}{{\partial}}
\newcommand{\dbar}{\bar{\partial}}
\newcommand{\1}{\mathbf{1}}

\newcommand{\bigO}{\mathcal{O}}

\newcommand{\dist}{\operatorname{dist}}
\newcommand{\supp}{\operatorname{supp}}

\newcommand{\lnorm}{\left\|}
\newcommand{\rnorm}{\right\|}

\def\norm#1{\lnorm {#1} \rnorm}

\makeatletter
\newcommand*\bigcdot{\mathpalette\bigcdot@{.5}}
\newcommand*\bigcdot@[2]{\mathbin{\vcenter{\hbox{\scalebox{#2}{$\m@th#1\bullet$}}}}}
\makeatother

\usepackage{tikz}
\usepackage{pict2e}
\usetikzlibrary{arrows,calc,chains, positioning, shapes.geometric,shapes.symbols,decorations.markings,arrows.meta}
\usetikzlibrary{patterns,angles,quotes}
\usetikzlibrary{decorations.markings}
\tikzset{middlearrow/.style={
			decoration={markings,
				mark= at position 0.6 with {\arrow{#1}} ,
			},
			postaction={decorate}
		}
	}
\tikzset{->-/.style={decoration={
				markings,
				mark=at position #1 with {\arrow{latex}}},postaction={decorate}}}
	
\tikzset{-<-/.style={decoration={
				markings,
				mark=at position #1 with {\arrowreversed{latex}}},postaction={decorate}}}
				
				\tikzset{
	master/.style={
		execute at end picture={
			\coordinate (lower right) at (current bounding box.south east);
			\coordinate (upper left) at (current bounding box.north west);
		}
	},
	slave/.style={
		execute at end picture={
			\pgfresetboundingbox
			\path (upper left) rectangle (lower right);
		}
	}
}
\tikzset{cross/.style={cross out, draw,
         minimum size=2*(#1-\pgflinewidth),
         inner sep=0pt, outer sep=0pt}}

\def\XXint#1#2#3{{\setbox0=\hbox{$#1{#2#3}{\int}$}
\vcenter{\hbox{$#2#3$}}\kern-.5\wd0}}

\allowdisplaybreaks

\begin{document}

\title[Free energy in the random normal matrix model]{Free energy and fluctuations in the random normal matrix model with spectral gaps}

\author{Yacin Ameur}
\address{Yacin Ameur\\
Department of Mathematics\\
Lund University\\
22100 Lund, Sweden}
\email{ Yacin.Ameur@math.lu.se}

\author{Christophe Charlier}
\address{Christophe Charlier\\
Department of Mathematics\\
Lund University\\
22100 Lund, Sweden}
\email{Christophe.Charlier@math.lu.se}

\author{Joakim Cronvall}
\address{Joakim Cronvall\\
Department of Mathematics\\
Lund University\\
22100 Lund, Sweden}
\email{Joakim.Cronvall@math.lu.se}



\begin{abstract} We study large $n$ expansions for the partition function of a Coulomb gas
\begin{equation*}Z_n=\frac 1 {\pi^n}\int_{\mathbb{C}^n}\prod_{1\le i<j\le n}|z_i-z_j|^2\prod_{i=1}^n e^{-nQ(z_i)}\, d^2 z_i,\end{equation*}
where $Q$ is a radially symmetric confining potential on the complex plane $\mathbb{C}$.

The droplet is not assumed to be connected, but may consist of a number of disjoint annuli and possibly a central disk. The boundary condition is ``soft edge'', i.e.,  $Q$ is smooth in a $\mathbb{C}$-neighbourhood of the droplet.

We include the following possibilities: (i) existence of ``outposts'', i.e., components of the coincidence set which fall outside of the droplet,
(ii) a Fisher-Hartwig singularity at the origin, (iii) perturbations $Q-\frac h n$ where $h$ is a smooth radially symmetric test-function.

In each case, the free energy $\log Z_n$ admits
a large $n$ expansion of the form
\begin{equation*}\log Z_n=C_1n^2+C_2n\log n+C_3 n+C_4\log n+C_5+\mathcal{G}_{n}+o(1)\end{equation*}
where $C_1,\ldots,C_5$ are certain geometric functionals. The $n$-dependent term $\mathcal{G}_n$ is bounded as $n\to\infty$; it arises in the presence of spectral gaps.

We use the free energy expansions to study the distribution of fluctuations of linear statistics. We prove that the fluctuations are well approximated by the sum of a Gaussian and certain independent terms which provide the displacement of particles from one component to another. This displacement depends on $n$ and is expressed in terms of the Heine distribution. We also prove (under suitable assumptions) that the number of particles which fall near a spectral outpost converges to a Heine distribution.
\end{abstract}

\maketitle

\section{Introduction}\label{intro}

\subsection{Potential-theoretic setup} \label{sub1} It is convenient to define some key objects which are used throughout. We will keep it brief and refer to \cite{ST} as
a general source for the potential theory, and to the introduction to \cite{ACC} for more detailed background and citations.

To begin with, we fix a function $Q:\C \to\R\cup\{+\infty\}$ called the \textit{external potential}, which is radially symmetric, i.e., $Q(z)=Q(|z|)$, and confining in the sense that
$\liminf_{z\to\infty}\left(Q(z)/\log|z|\right)>2.$

Assuming
that $Q$ is lower semicontinuous and finite on some set of positive capacity, Frostman's theorem implies the existence of a unique \textit{equilibrium measure}
$\sigma$, which minimizes the weighted logarithmic energy
\begin{equation*}I_Q[\mu]=\int_{\C^2}\log\frac 1 {|z-w|}\, d\mu(z)\,d\mu(w)+\int_\C Q\, d\mu,\end{equation*}
over all compactly supported Borel probability measures $\mu$. The support of the equilibrium measure is called the \textit{droplet}  and is denoted $S=S[Q]:=\supp \sigma$.

\begin{figure}
\begin{tikzpicture}[master]
\node at (0,0) {\includegraphics[width=5.7cm]{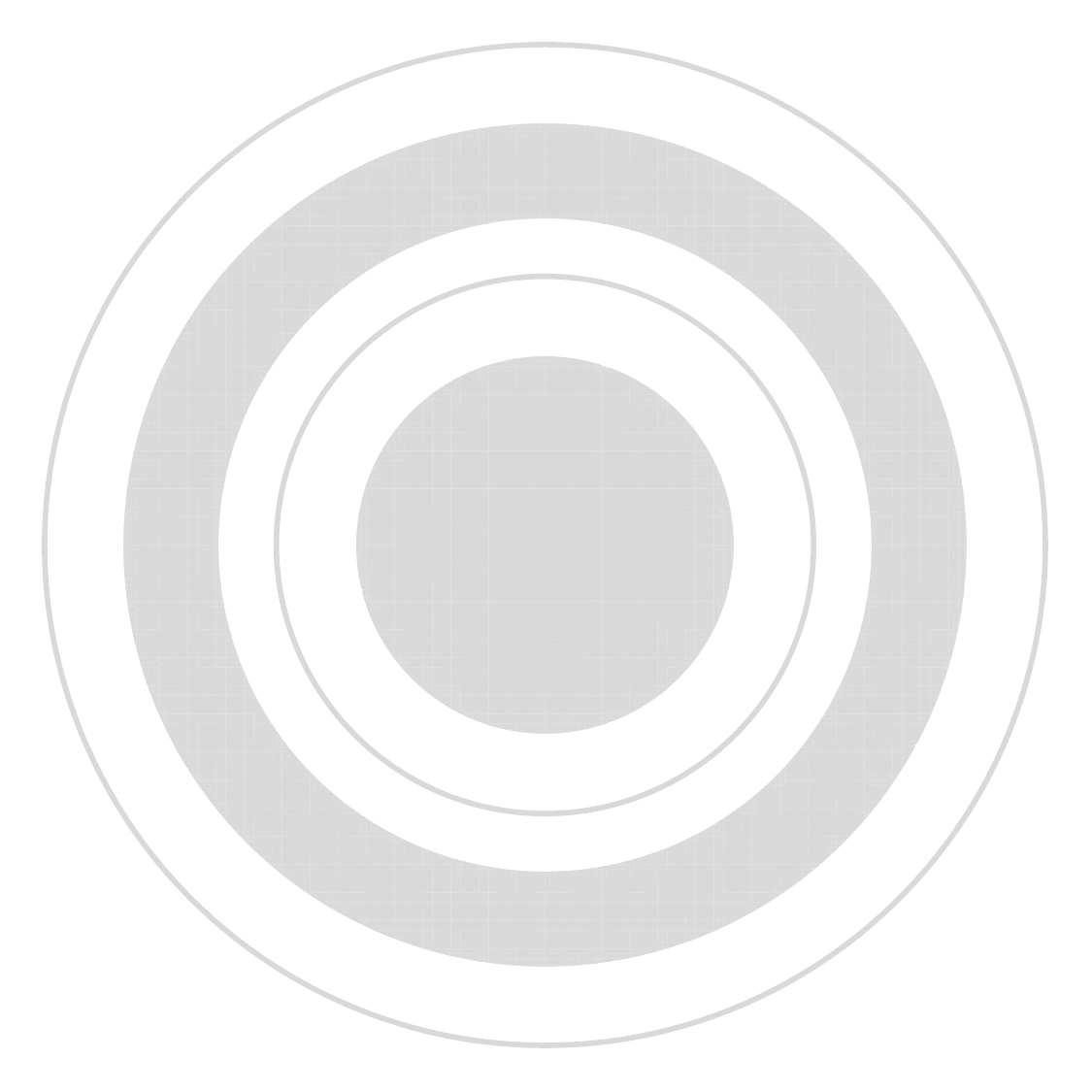}};

\draw[fill] (0,0) circle (0.02);
\draw[fill] (0.984,0) circle (0.02);
\draw[fill] (1.7,0) circle (0.02);
\draw[fill] (2.2,0) circle (0.02);
\node at (0,0.15) {\tiny $a_{0}$};
\node at (0.984,0.15) {\tiny $b_{0}$};
\node at (1.7,0.15) {\tiny $a_{1}$};
\node at (2.2,0.15) {\tiny $b_{1}$};

\node at (0,-3) {\small $a_{0}=0$};
\end{tikzpicture}\hspace{0.45cm}
\begin{tikzpicture}[slave]
\node at (0,0) {\includegraphics[width=5.7cm]{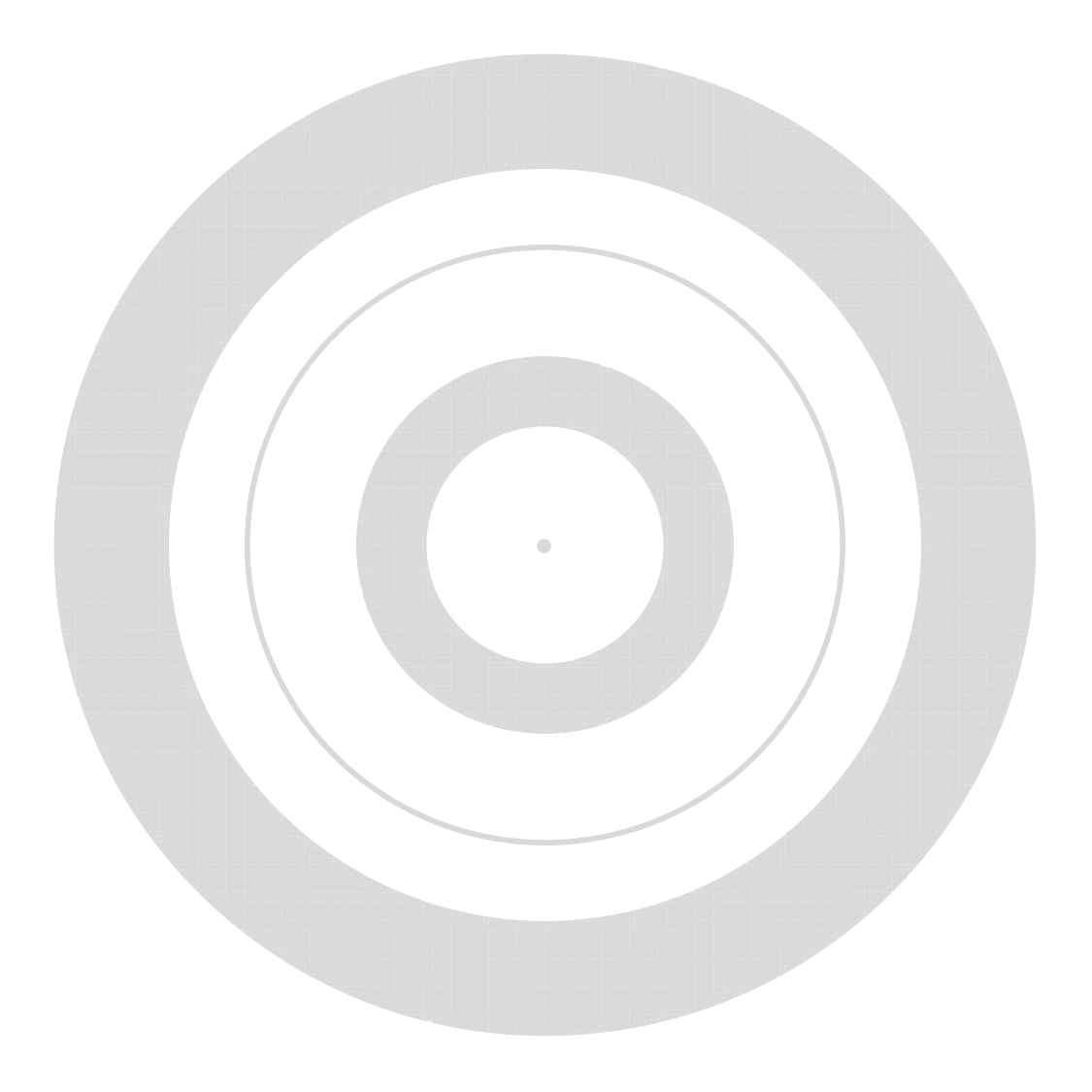}};
\draw[fill] (0.615,0) circle (0.02);
\draw[fill] (0.984,0) circle (0.02);
\draw[fill] (1.955,0) circle (0.02);
\draw[fill] (2.56,0) circle (0.02);
\node at (0.615,0.15) {\tiny $a_{0}$};
\node at (0.984,0.15) {\tiny $b_{0}$};
\node at (1.955,0.15) {\tiny $a_{1}$};
\node at (2.56,0.15) {\tiny $b_{1}$};

\node at (0,-3) {\small $a_{0}>0$};
\end{tikzpicture}
\caption{\label{fig:droplets} $N=1$ and $S^{*}$ is shown in gray. Left: Central disk droplet with two outposts.
Right: Annular droplet with two outposts.}
\end{figure}
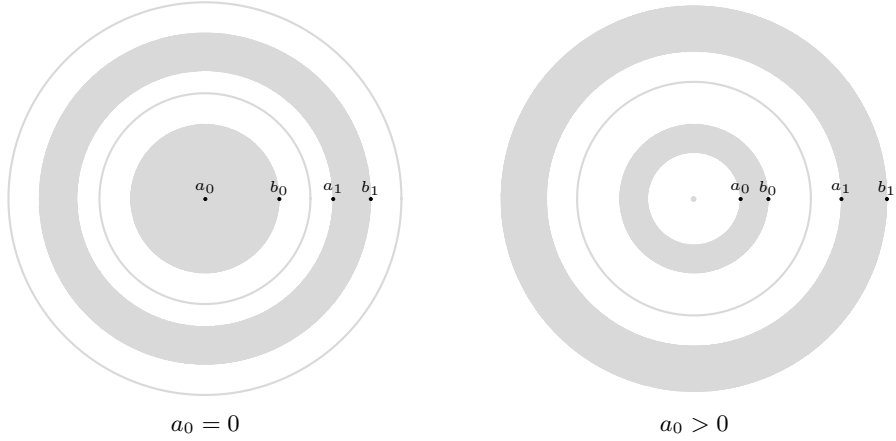

If $Q$ is $C^2$-smooth in a neighbourhood of $S$, then $\sigma$ is absolutely continuous with respect to the
normalized Lebesgue measure
$dA(z):=\frac 1 \pi\, dxdy$, $(z=x+iy)$,
and takes the form
\begin{equation*}
d\sigma(z)=\Delta Q(z)\,\1_S(z)\, dA(z),\end{equation*}
where we write $\Delta:=\d\dbar=\frac 1 4 (\d_{xx}+\d_{yy})$ and $\1_S$ for the characteristic function of $S$.

The \textit{obstacle function} $\check{Q}(z)$ (with respect to the obstacle $Q$) is by definition the pointwise supremum of the set of subharmonic functions $s(z)$ on $\C$ such that $s\le Q$ everywhere and $s(w)\le 2\log|w|+\bigO(1)$ as $w\to\infty$. It is easy to check that $\check{Q}$ is globally subharmonic and $C^{1,1}$-smooth with $\check{Q}\le Q$ everywhere and
$\check{Q}(z)=2\log|z|+\bigO(1)$ as $z\to\infty$.

The \textit{coincidence set} $S^*=S^*[Q]$ for the obstacle problem is defined by
\begin{equation*}
S^*:=\{z\in\C\,;\, Q(z)=\check{Q}(z)\}.
\end{equation*}

\textit{We assume throughout the rest of the paper that $Q$ is $C^6$-smooth in a neighbourhood of $S^*$.}

Evidently $S$ and $S^*$ are compact sets, $\Delta Q\ge 0$ at each point of $S^*$, and $\sigma(S^*\setminus S)=0$. We refer to points of
$S^*\setminus S$ as \textit{shallow points}, and call a connected component of $S^*\setminus S$ an \textit{outpost} of the droplet.

A connected component of $S^*$ is either a disk $\D_b=\{z\,;\, |z|\le b\}$, an annulus $A(a,b)=\{z\,;\,a\le |z|\le b\}$, the singleton $\{0\}$, or a circle $\{{z\,;\,}|z|=t\}$.
With a mild restriction, we will assume that $S^*$ has only finitely many connected components.
Thus
\begin{equation}\label{droplet}
S=\bigcup_{\nu=0}^N A(a_\nu,b_\nu),
\end{equation}
where $0\le a_0<b_0<a_1<b_1<\cdots<a_N<b_N$, and $S^*$ is obtained by possibly adjoining finitely many outposts $\{z\,;\,|z|=t_m\}$ where $t_m\ge 0$ (see Figure \ref{fig:droplets}). Throughout, we will use the notation $S^{\nu}:=A(a_\nu,b_\nu)$, $\nu=0,\ldots,N$.

In what follows, we always assume that we have the \textit{strict} subharmonicity $\Delta Q>0$ in a neighbourhood of $S^*$.

We distinguish between two possibilities for the topology, in terms of the \textit{Euler characteristic} $\chi(S)$:
\begin{enumerate}[label=(\roman*)]
\item $a_0=0$: ``Central disk droplet''. In this case $\chi(S)=1+ N\cdot 0 = 1$.
\item $a_0>0$: ``Annular droplet''. In this case $\chi(S) = (N+1)\cdot 0 = 0$.
\end{enumerate}

\subsection{Coulomb gas ensembles} \label{cge} Besides $Q$ we shall consider perturbations of the form
\begin{equation}\label{pert}
\tilde{Q}(z):= Q(z)-\frac s nh(z) -\frac {\alpha} n\ell(z),
\end{equation}
where $\ell$ is the logarithm
\begin{equation}\label{elldeff}
\ell(z):=2\log |z|,
\end{equation}
while $s,\alpha$ are real parameters with $\alpha>-1$.
We assume throughout that the test function $h(z)$ in \eqref{pert} is radially symmetric and $C^4$-smooth in a neighbourhood of $S^*$.
(\footnote{The behaviour outside of such a neighbourhood is less sensitive and also less interesting; for example requiring that $h$ is continuous and bounded on $\C$ will do. For simplicity, we adopt this convention in what follows.})

By definition, the partition function of
the Coulomb gas $\{z_j\}_1^n$ in
potential \eqref{pert}, at an inverse temperature $\beta>0$, is the $n$-fold integral
\begin{equation}\label{pf1}
Z_n^\beta=Z_n^\beta[\tilde{Q}]=\int_{\C^n} |\Delta_n(z)|^\beta e^{-n\sum_{j=1}^n \tilde{Q}(z_j)}\, dA_n(z)
\end{equation}
where $\Delta_n(z)=\prod_{j>k}(z_j-z_k)$ is the Vandermonde determinant and $dA_n=(dA)^{\otimes n}$ is the normalized Lebesgue measure on $\C^n$.

In the following we set $\beta=2$, and  we write
$Z_n:=Z_{n,sh}^{(\alpha)}$ for the corresponding partition function \eqref{pf1}.

The Coulomb gas in potential $\tilde{Q}$ (at $\beta=2$) is a random sample $\{z_j\}_1^n$ drawn with respect to the following Gibbs distribution on $\C^n$,
\begin{equation}\label{gibbs0}d\tilde{\Prob}_n(z)=\frac 1 {Z_n}|\Delta_n(z)|^2 e^{-n\sum_{j=1}^n \tilde{Q}(z_j)}\, dA_n(z).\end{equation}

In the case $s=0$ we drop the tildes and write $\mathbb{P}_n$ or $\mathbb{P}_n^{(\alpha)}$ depending on whether or not $\alpha=0$.

\smallskip

Inspired by recent progress in \cite{BKS} we shall prove large $n$ expansions for $\log Z_n$ in a variety of new situations.  These expansions depend on a number of ``geometric functionals''
which we now list, for future convenience.

\subsubsection{Geometric functionals} In what follows, $Q$ always refers to the ``unperturbed'' potential,
while $\tilde{Q}$ is the perturbation \eqref{pert}; the measure $d\sigma=\1_S\cdot\Delta Q\, dA$ is the equilibrium measure in
potential $Q$.

We will make use of the following items.

\begin{enumerate}[label=(\Roman*)]
\item The weighted \textit{energy} of the equilibrium measure:
\begin{align}\label{Q-energy}
I_Q[\sigma]=\int_{\C^2}\log\frac 1 {|z-w|}\, d\sigma(z)\, d\sigma(w)+\int_\C Q\, d\sigma.
\end{align}
\item The (negative of the) \textit{entropy}
of the equilibrium measure:
$$E_Q[\sigma]=\int_\C \log\Delta Q\, d\sigma.$$

\item The \textit{Euler characteristic} $\chi(S)$ of the droplet (see above).
\item We now define a geometric functional $F_Q[\sigma]$.
By a slight abuse of notation, this functional is the sum
$F_Q[\sigma]=\sum_{C} F_Q[C]$ where $C$ ranges over the connected components
of $S$. For each component $C$ of $S$, $F_Q[C]$ is defined as follows:

(i) If $C$ is the annulus $A(a,b)$ where $0<a<b$, then we define
\begin{equation}\label{Fq_ann}\begin{split}F_Q[A(a,b)]&:=\frac 1 {12}\log\left[\frac {a^2\Delta Q(a)}{b^2\Delta Q(b)}\right]
-\frac 1 {16}\left[b\frac {\d_r\Delta Q(b)}{\Delta Q(b)}
 -a\frac {\d_r\Delta Q(a)}{\Delta Q(a)}\right]\\
 &\qquad +\frac 1 {24}\int_a^b\left[\frac {\d_r\Delta Q(r)}{\Delta Q(r)}\right]^{\,2}\, r\, dr.\\ \end{split}\end{equation}

(ii) If $C$ is the disc $\D_b$ then we define $F_Q[C]$ by
\begin{equation}\label{Fq_disc}F_Q[\D_b]:=\frac 1 {12}\log\left[\frac 1 {b^2\Delta Q(b)}\right]-\frac 1 {16}b\frac {\d_r\Delta Q(b)}{\Delta Q(b)}+\frac 1 {24}\int_0^b
\left[\frac {\d_r\Delta Q(r)}{\Delta Q(r)}\right]^{\, 2}\, r\, dr.
\end{equation}

\item \label{bgdef} The \textit{Barnes $G$-function} $G(z)$, see e.g.~ \cite[Section 5.17]{NIST} or \cite{AAR}.

\item \label{crind} The \textit{cumulative masses} $M_\nu$ for $-1\le \nu\le N$ are defined by  $M_{-1}=0$ and
\begin{equation}\label{mass}
M_\nu=\sigma(\{|z|\le b_\nu\}),\qquad \nu=0,\ldots,N.
\end{equation}
 We define $x_\nu=x_\nu(n)$ to be the fractional part
$$x_\nu=\{M_\nu n\}:=M_\nu n-m_\nu$$
where the \textit{critical index} $m_\nu$ is the integer part $m_\nu:=\lfloor M_\nu n\rfloor.$ (Note that $M_N=1$.)
\item \label{rmdef} For $0\le \nu\le N-1$, we define constants $\rho_\nu$, $\theta_{\nu,\alpha}=\theta_{\nu,\alpha}(n)$, $c_\nu=c_\nu(h)$ by
\begin{equation}\label{rthdef}\rho_\nu := \frac {b_{\nu}}{a_{\nu+1}},\qquad \theta_{\nu,\alpha}:=\sqrt{\frac {\Delta Q(b_\nu)}{\Delta Q(a_{\nu+1})}}\, \rho_\nu^{2(x_\nu-\alpha)},\qquad c_\nu:=h(a_{\nu+1})-h(b_\nu),\end{equation}
and introduce the shorthand notation
\begin{align}\label{mudef}\mu_\nu=\mu_\nu(s,\alpha;n,h):=\theta_{\nu,\alpha}e^{sc_\nu}.
\end{align}

\item  For fixed $q$ with $0<q<1$ and $z\in\C$, the \emph{q-shifted limiting factorial} is the product (\cite{AAR,GR})
\begin{equation}\label{limpoch}
(z;q)_\infty := \prod_{i=0}^{\infty} (1-zq^i).
\end{equation}

\item \label{oscndef} We define the $n$:th \textit{net displacement term} $\Osc_n(s,\alpha)$ in terms of \eqref{rthdef} -- \eqref{limpoch} by
\begin{align}
\Osc_n(s,\alpha):=&\sum\limits_{\nu = 0}^{N-1} (x_\nu \log \mu_\nu-x_\nu^2\log\rho_\nu)
 +\sum\limits_{\nu = 0}^{N-1} \log[(-\rho_\nu \mu_\nu\,;\, \rho_\nu^2)_\infty] +\sum\limits_{\nu = 0}^{N-1}\log[(-\frac {\rho_\nu}{\mu_\nu}\,;\,\rho_\nu^2)_\infty].\label{Oscn}
\end{align}
\item With the test function $h$ in \eqref{pert} we also associate the quantities
\begin{align}
\label{eh} e_h & =\frac{1}{2}\int_{S} h\, \Delta \log\Delta Q\, dA +  \frac{1}{8\pi}\int_{\partial S} \partial_n h(z)\, |dz|-\frac{1}{8\pi}\int_{\partial S} h(z) \frac{\partial_n \Delta Q(z)}{\Delta Q(z)}\, |dz|,
\\
\label{vh} v_h&= \frac{1}{4} \int_{S} |\nabla h(z)|^2\, dA(z).
\end{align}
Here and throughout, ``$\d_n$'' designates differentiation in the normal direction to $\d S$ pointing out from the droplet $S$
\end{enumerate}

\begin{rem} The terms $F_Q[C]$ in \eqref{Fq_ann} and \eqref{Fq_disc} have interpretations in terms of $\zeta$-regularized determinants for certain pseudodifferential operators
(Laplacians and Neumann's jump) via the Polyakov-Alvarez formula, see \cite{BKS,HZ,JV,Ka,Kl,W,ZW} and in particular \cite[Appendix C]{ZW}.

In the case when $S$ is connected, the terms $e_h$ and $v_h$ can be interpreted as the limiting expectation and variance of linear statistics, as proven in \cite{AHM}. In \cite[Theorems 1.4 and 6.1]{ACC} it is noted that the same interpretation holds in the case of disconnected droplets, provided that $h$ is the real part of an analytic function in each spectral gap.

In connection to the term in (II), we remark that the concept of entropy of a probability measure has a well-known interpretation in information theory, see e.g.~Khinchin's book \cite{K}.

\end{rem}

\begin{rem} The displacement term \eqref{Oscn} can be rewritten in terms of the
Jacobi theta function
\begin{align}\label{Jacobi theta}
\theta(z;\tau) = \sum_{\ell=-\infty}^{+\infty} e^{2 \pi i \ell z}e^{\pi i \ell^{2} \tau}.
\end{align}
To see this, note that
\begin{align*}
\Osc_n(s,\alpha) = \sum_{\nu=0}^{N-1} \Theta(x_{\nu};\rho_{\nu},\rho_{\nu}^{1-2x_{\nu}}\mu_{\nu}),
\end{align*}
where the function $\Theta(x;p,q)$ is given by
\begin{equation}\label{mjac}
\Theta=\Theta(x;p,q) := x(x-1)\log p+x\log q+ \log[(-qp^{2x};p^2)_\infty] + \log[(-q^{-1}p^{2(1-x)};p^2)_\infty],
\end{equation}
and by \cite[Lemma 3.25]{C}, $\Theta$ is related to the logarithm of \eqref{Jacobi theta} by
\begin{equation*}\begin{split}
\Theta= \frac{1}{2}\log \bigg( \frac{\pi q p^{-\frac{1}{2}}}{\log(p^{-1})} \bigg) &+ \frac{(\log q)^{2}}{4\log(p^{-1})} - \sum_{j=1}^{+\infty} \log(1-p^{2j})+ \log \theta \bigg( x +\frac 1 2+\frac 1 2   \frac{\log q}{\log p}\, ;\, \frac{\pi i}{\log(p^{-1})} \bigg).
\end{split}\end{equation*}

Previous works on spectral gaps have typically been formulated in terms of \eqref{Jacobi theta}, see e.g.~\cite{ACC,C,CFWW} and the references therein. However, in the present work, the form \eqref{Oscn} is advantageous, as it more directly relates to the Heine distribution (to be introduced below).

\end{rem}

\subsection{The regular case}
We begin by considering the case where $\alpha=0$ (no conical singularity) and $S=S^*$ (no shallow points).
In this case, it is convenient to write
\begin{align}\label{Fn} \Osc_n&(s):=\Osc_n(s,0).\end{align}

We have the following result.

\begin{thm} \label{regular_expansion}
Assume that $S=S^*$ and $\alpha=0$.
Then for all real $s$, $|s|\leq \log n$, we have, as $n\to\infty$,
	\begin{equation}\label{lin2}\begin{split}
\log Z_{n,sh} = -n^2 I_Q[\sigma] &+ \frac 1 2 n\log n + n\left[ \frac{\log (2\pi)}{2} -1 -\frac{E_{Q}[\sigma]}2  +s\int\limits_\C h\, d\sigma \right]
 + \frac {6-\chi(S)}{12} \log n\\
&	+\chi(S)\, \zeta'(-1)+ F_Q[\sigma]+\frac{\log(2\pi)}2 +se_h+\frac {s^2} 2 v_h+\Osc_n(s) + \calE(n).\\
\end{split}
	\end{equation}

The error term $\calE(n)$ is uniform in $s$ and satisfies $\calE(n)=\bigO(\frac{1+s^{2} }{n})$ if $a_0>0$, while
 $\calE(n)=\bigO\Big(\frac{(\log n)^3}{n^\frac{1}{12}}\Big)$ if
$a_0=0$.
\end{thm}

In the special case when $Q$ is 
smooth and strictly subharmonic on the punctured plane $\C\setminus\{0\}$ (which implies that $S$ is connected) and $s=0$, Theorem \ref{regular_expansion} is precisely the union of  \cite[Theorem 1.1(i) and Theorem 1.2(i)]{BKS}. By contrast, if $N\ge 1$, then there are always points (in the gaps) where $Q$ is strictly superharmonic. (See for example the computations in Subsection \ref{laplace_me} below.)

\begin{rem} For $s=0$, our result \eqref{lin2} gives the optimal error term $\bigO(1/n)$ if $a_0>0$. 
\end{rem}

\subsubsection{Fluctuations} A well-known principle dictates that fluctuations of smooth linear statistics are determined by the $\bigO(1)$-term of the large $n$ expansion of the free energy $\log Z_{n,sh}$. We shall exploit
this principle in a variety of situations, to deduce fluctuation theorems.

The key idea is the following. Given a suitable test function $h(z)$ and a random sample $\{z_j\}_1^n$ with respect to the unperturbed Gibbs distribution $\Prob_n$ in potential $Q$ (see \eqref{gibbs0} with $\alpha=s=0$)
we define the
random variable
\begin{align}\label{lins}
\fluct_n h &= \sum\limits_{j=1}^{n} h(z_j) -n\int\limits_{\C} h\, d\sigma.\end{align}

Given any random variable $X$, we write
$F_X(s)=\log\E(e^{sX})$ for its cumulant generating function (in short: CGF). We recall from basic probability theory that if the moment problem for $X$ is determinate, and if $F_{X_n}(s)\to F_X(s)$ uniformly in some neighbourhood of $s=0$ as $n\to\infty$, then $X_n$ converges in distribution to $X$. (See for instance \cite[Section 30]{Bill}.)

 The normal $N(\mu,\sigma^2)$-distribution is determined by the CGF $F_X(s)=s\mu+\frac {s^2} 2 \sigma^2$.

With this in mind, we pass to the CGF of \eqref{lins}, which we denote
\begin{equation}\label{CGF}F_{n,h}(s):=\log\E_n(e^{s\fluct_n h}).\end{equation}

By a standard computation (see e.g. \cite[Section 2.1]{ACC}) we have
the following formula for the derivative
\begin{equation}\label{diffint} F_{n,h}' (s) = \frac{d}{ds} \log Z_{n,sh} -n\int\limits_{\C} h\, d\sigma.\end{equation}

Integrating both sides from $0$ to $s$, using $F_{n,h}(0)=0$ and Theorem \ref{regular_expansion}, we obtain:

\begin{cor} \label{fluct_reg}
As $n\to\infty$, uniformly for $|s|\le \log n$
	\begin{align}\label{veh}
	F_{n,h}(s)& = se_h+\frac {s^2} 2 v_h+\Osc_n(s)-\Osc_n(0) + \calE(n),
\end{align}
where $\calE(n)$ satisfies the same estimates as in Theorem \ref{regular_expansion}.
\end{cor}

Except for the succinct formulation and the order of the error term, Corollary \ref{fluct_reg} is not new.
Indeed, it alternatively follows by combining the recent boundary fluctuation theorem in \cite[Theorem 1.6]{ACC} with an earlier result for the bulk from \cite{AHM0}.
However, our present derivation using the partition function is new. The analysis in \cite{ACC} (which is of independent interest) is based around a local analysis of the one-point function (in potential $\tilde{Q}$) in various regimes of the plane.

\subsubsection{Probabilistic interpretation}
We shall now give an
interpretation of Corollary \ref{fluct_reg} in terms of a probability distribution on the non-negative integers $\Z_+=\{0,1,2,\ldots\}$, known as the \textit{Heine distribution} \cite{Ke}.

\begin{defn}
Let $\theta$ be a positive real parameter and $q$ a number with $0<q<1$.
A $\Z_+$-valued random variable $X$  is said to have a Heine distribution with parameters $(\theta,q)$, denoted $X\sim\He(\theta,q)$, if
\begin{equation}\label{pekka}
\Prob(\{X=k\})=\frac 1 {(-\theta;q)_\infty}\frac {q^{\frac 1 2 k(k-1)} \theta^k}{(q;q)_k},\qquad k\in\Z_+,
\end{equation}
where we follow the notation of \cite{GR} for $q$-shifted factorials,
\begin{align*}(z;q)_k=\prod_{i=0}^{k-1}(1-zq^i),\qquad (z;q)_\infty=\prod_{i=0}^\infty (1-zq^i).\end{align*}
\end{defn}

To see that the numbers in \eqref{pekka} sum to one, we use the following identity, which is a special case of the $q$-binomial Theorem  (see~\cite{AAR,GR,FK,MH})
	\begin{equation}\label{qbi}
	(z;q)_n = \sum\limits_{k=0}^{n} \frac{(q;q)_n}{(q;q)_k (q;q)_{n-k}} q^{\frac 12 k(k-1)}(-z)^k.
	\end{equation}
	
Taking the limit in \eqref{qbi} as $n\to\infty$ we get
	$$
	(z;q)_\infty = \sum\limits_{k=0}^{\infty} \frac{q^{\frac 12k(k-1)}(-z)^k}{(q;q)_k},
	$$
proving that the Heine distribution is indeed a probability distribution.

For later application, we note that by letting $z=-\theta \sqrt{q} e^{cs}$ in the last relation, we obtain	
\begin{equation}\label{euler_id}\sum_{k=0}^\infty\frac {q^{\frac 1 2 k^2}}{(q;q)_k} \theta^ke^{csk} =\prod_{j=0}^\infty (1+\theta e^{cs} q^{\frac 1 2 (2j+1)}).
\end{equation}

We are now ready to formulate a probabilistic interpretation of Corollary \ref{fluct_reg}.

\begin{cor}\label{fluct_reg2} For for each $\nu$, $0\le \nu\le N-1$, let $\{X_\nu^+,X_\nu^-\}$
be a set of independent, Heine distributed random variables with parameters
$$X_\nu^+\sim \He(\theta_\nu\rho_\nu,\rho_\nu^2),\qquad
X_\nu^{-}\sim \He(\theta_\nu^{-1}\rho_\nu,\rho_\nu^2),$$
where $\theta_\nu:=\theta_{\nu,0}$ is given in \eqref{rthdef}.

Also let $W$ be a Gaussian $N(e_h,v_h)$-distributed random variable independent from all $X_\nu^{\pm}$, i.e.,
$F_W(s)  =se_h+\frac {s^2} 2 v_h,$
and define $K_n=\sum_{\nu=0}^{N-1}c_\nu x_\nu$, where $c_\nu$ are defined in \eqref{rthdef}.

Then, as $n \to \infty$, uniformly for $|s| \leq \log n$,
$$F_{n,h}(s)=sK_n+F_W(s)+\sum_{\nu=0}^{N-1}(F_{c_\nu X_\nu^+}(s)+F_{-c_\nu X_\nu^-}(s))+\calE(n),$$
where $\calE(n)$ satisfies the same estimates as in Theorem \ref{regular_expansion}.
\end{cor}

\begin{proof} Let $X\sim\He(\theta\sqrt{q},q)$ and $c$ a real constant. By \eqref{euler_id}
the moment generating function of $cX$ is
\begin{equation}\label{mgfc}\mathbb{E}(e^{csX})=\frac {(-\theta\sqrt{q}e^{cs};q)_\infty} {(-\theta\sqrt{q};q)_\infty}.
\end{equation}
Now define
$\mu_\nu(s):=\mu_\nu(s,0;n,h)$ (see \eqref{mudef}).
Then
\begin{align}
\Osc_n(s)-\Osc_n(0)=\sum_{\nu=0}^{N-1}\Big[x_\nu&sc_\nu+ \log \frac{(-\rho_\nu\mu_\nu(s); \rho_\nu^2)_\infty}{(-\rho_\nu\mu_\nu(0); \rho_\nu^2)_\infty}  + \log \frac{(-\rho_\nu/\mu_\nu(s); \rho_\nu^2)_\infty}{(-\rho_\nu/\mu_\nu(0); \rho_\nu^2)_\infty} \Big]. \label{lol1}
\end{align}
On the other hand, by comparing with \eqref{mgfc} (with $q=\rho_\nu^2$) and recalling that $\mu_\nu(s)=\mu_\nu(0)e^{sc_\nu}$, we infer that if $X_\nu^+\sim\He(\theta_\nu\rho_\nu,\rho_\nu^2)$, then
$$F_{c_\nu X_\nu^+}(s):=\log\mathbb{E}(e^{sc_\nu X_\nu^+})=\log \frac{(-\rho_\nu\mu_\nu(s); \rho_\nu^2)_\infty}{(-\rho_\nu\mu_\nu(0); \rho_\nu^2)_\infty}.$$
Similarly, if $X_\nu^-\sim\He(\theta_\nu^{-1}\rho_\nu,\rho_\nu^2)$, then
$$F_{-c_\nu X_\nu^-}(s)= \log \frac{(-\rho_\nu/\mu_\nu(s); \rho_\nu^2)_\infty}{(-\rho_\nu/\mu_\nu(0); \rho_\nu^2)_\infty}.
$$
Substituting the above in \eqref{lol1} and recalling \eqref{veh}, we finish the proof of the corollary.
\end{proof}

\begin{rem} The random variables $W=W(h)$ in Corollary \ref{fluct_reg2} form a Gaussian field in $\C$ generalizing the one appearing for connected droplets in \cite{AHM}.
In the present case, we need to subtract the sum $K_n+\sum_0^{N-1}c_\nu(X_\nu^+-X_\nu^-)$ from $\fluct_n h$ in order to obtain convergence to the Gaussian field.

It is natural to think of $X_\nu^+$ as a displacement of particles from $S^{\nu}$ to $S^{\nu+1}$ and $X_{\nu}^{-}$ as an independent displacement from $S^{\nu+1}$ to $S^\nu$. Note that the constant $K_n$ as well as the distributions of the displacement terms $X_\nu^{\pm}$ depend on $n$ through $x_\nu=\{M_\nu n\}$.

Recall from \cite{Ke2,ACC} that an integer valued random variable $Y$ is said to have a \textit{discrete normal distribution} with parameters $\lambda>0$ and $0<q<1$ (denoted  ``$Y\sim dN(\lambda, q)$'') if it has the probability function
$$
\mathbb{P}(\{Y=k\}) = C_{\lambda,q}\lambda^k q^{\frac 12k(k-1)},\qquad (k\in\Z),
$$
where $C_{\lambda,q}$ is the normalizing constant.
In this notation, it follows from \cite[Theorem 2]{Ke2} that
$$X_\nu^+ - X_\nu^- \sim dN(\theta_{\nu} \rho_{\nu}, \rho_{\nu}^2).$$
\end{rem}

\subsection{Fisher-Hartwig singularity} \label{fssub} We now consider the partition function $Z_{n,sh}^{(\alpha)}$ where $\alpha>-1$ is the parameter in \eqref{pert}. (Physically, this corresponds to insertion of a point-charge of ``strength'' $\alpha$, see e.g.~\cite{AHM0}. In Hermitian random matrix theory, such singularities are usually referred to as root-type Fisher-Hartwig singularities, see e.g. \cite{CFWW} and the references therein.)

We maintain the assumption that the droplet takes the form \eqref{droplet}, and that there are no shallow points, i.e., $S=S^*$. We shall focus on the case of a central disk droplet,
i.e., we assume that $a_0=0$. (Annular droplets corresponding to $a_{0}>0$
are uninteresting in this connection, since the logarithmic term $\ell(z)=2\log|z|$ is then smooth in a neighbourhood of the droplet.)

Recall that $e_{h}$ and $v_{h}$ are defined by \eqref{eh} and \eqref{vh}, respectively.
Given a suitable function $f$, we also define
\begin{align}
v_{\nu,f} = \frac{1}{4} \int_{S^{\nu}} |\nabla f(z)|^2\, dA(z),
\end{align}
where $S^\nu=A(a_\nu,b_\nu)$ is the $\nu$:th connected component of $S$.

We shall frequently use the linear combination
\begin{equation}\label{kdef}
k(z)=k_{s,\alpha}(z):=sh(z)+\alpha\ell(z).
\end{equation}

For the following result, we note that the integral
$$v_{\nu,k}=\frac{1}{4}\int_{S^{\nu}}|\nabla k|^2\, dA
=\frac {s^2} 4 \int_{S^{\nu}} |\nabla h|^2\, dA+\frac {s \alpha} 2\int_{S^\nu}\nabla h\bigcdot \nabla \ell\, dA+\frac {\alpha^2} 4\int_{S^\nu}|\nabla\ell|^2\, dA$$
converges when $\nu\ge 1$ but diverges when $\nu=0$ (unless $\alpha=0$). (Here $x\bigcdot y=x_1y_1+x_2y_2$ denotes the usual dot product in $\R^2$.)

\begin{thm}\label{conical_expansion} Suppose that $S=S^*$ is a central disk droplet, i.e., $a_0=0$, and that $\alpha>-1$. Then for all real $s$  with $|s|\leq \log n$, we have, as $n\to\infty$,
\begin{align*}
\log Z_{n,sh}^{(\alpha)}  =& - n^2 I_Q[\sigma] + \frac{1}{2}n \log n  +  n \left[\frac{\log 2\pi}{2} -1 -\frac{E_{Q}[\sigma]}2  + \int_\C k \, d\sigma \right]   +\Big(\frac{5}{12}+\frac{\alpha^{2}}{2}\Big)\log n \nonumber\\
& +  \zeta'(-1) - \log G(1+\alpha) + F_Q[\sigma] + \frac{1+\alpha}{2}\log(2\pi)  + \, e_k -\frac{\alpha}{2}
+ \frac{s^2}{2}v_{0,h} + \frac{1}{2}\sum\limits_{\nu=1}^N v_{\nu,k}  \nonumber \\
 & + \frac{\alpha^{2}}{2} \log \big( b_{0}^{2}\Delta Q(0)  \big)    + \Osc_{n}(s,\alpha) +\alpha s \big( h(b_{0})-h(0) \big)  + \bigO\Big(\frac{(\log n)^3}{n^\frac{1}{12}}\Big),
\end{align*}
where $G(z)$ and $\Osc_n(s,\alpha)$ are defined in Section \ref{cge}, items \ref{bgdef} and \ref{oscndef}, respectively.
\end{thm}

\begin{rem} In dimension one,
large $n$ expansions in connection with Fisher-Hartwig singularities have been well studied, see e.g.~\cite{CG FH 2021, CFWW} and the references there. In dimension two, various aspects of root-type Fisher-Hartwig singularities and conical singularities are studied in e.g.~\cite{WW,Kl,AKS2, DS, BC FH 2022} and the references therein.
\end{rem}

We shall now use the relationship between the partition function and the cumulant generating function and draw two immediate conclusions.

First, note that using \eqref{diffint} with $\E_{n}$ replaced by $\E_{n}^{(\alpha)}$ and Theorem \ref{conical_expansion}, we obtain the following result. (The notation $\E_{n}^{(\alpha)}$ means that the expectation is taken with respect to \eqref{gibbs0} with $s=0$.)

\begin{cor}\label{cor1} For fixed $\alpha>-1$, the cumulant generating function $$F_{n,h}^{(\alpha)}(s):=\log \E_n^{(\alpha)} \exp(s\fluct_n h)$$
	 satisfies the large $n$ expansion
	$$
	F_{n,h}^{(\alpha)}(s) = se_{h,\alpha}  + \frac{s^2}{2} v_{h}
+ \Osc_n(s,\alpha)-\Osc_n(0,\alpha)  +\bigO\Big(\frac{(\log n)^3}{n^\frac{1}{12}}\Big),
	$$
where $e_{h,\alpha}:=e_h+\frac \alpha 2\sum_{\nu=1}^N
\int_{S^\nu}\nabla h\bigcdot \nabla \ell\, dA+\alpha(h(b_0)-h(0))$.
\end{cor}

Using \eqref{diffint} with $h(z)$ replaced by $\ell(z)=2\log|z|$ and using Theorem \ref{conical_expansion}, we also obtain the following result.

\begin{cor} \label{cor2}
Let, for $\alpha>-1$,
$$F_{n,\ell}(\alpha):=\log \E_n \exp(\alpha\fluct_n\ell),$$
be the cumulant generating function of $\fluct_n\ell$, where the expectation is with respect to
the Gibbs measure \eqref{gibbs0} for $s=\alpha=0$.	
Then as $n\to\infty$,
\begin{align*}F_{n,\ell}(\alpha)=\frac{\alpha^2} 2\log n+\alpha\tilde{e}_\ell+\frac {\alpha^2} 2\tilde{v}_\ell
-\log G(1+\alpha)+ \Osc_n(0,\alpha)-\Osc_n(0,0)+ \bigO\Big(\frac{(\log n)^3}{n^\frac{1}{12}}\Big),
\end{align*}
where
$$\tilde{e}_\ell:=e_\ell+\frac {\log (2\pi)} 2
,\qquad \tilde{v}_\ell:=\log(b_0^2\Delta Q(0))+\sum_{\nu=1}^N v_{\nu,\ell}.$$
\end{cor}

\begin{rem} Note that Corollary \ref{cor2} implies that the random variables \begin{equation}\label{logf}
\frac {\fluct_n\ell}  {\sqrt{\log n}}
\end{equation}
converge in distribution to a standard normal as $n\to\infty$. This kind of convergence is known to be ``universal'' for all suitable potentials giving rise to a
connected droplet whose interior contains the origin, see
\cite{AKS2,WW}.
More remarkably, we here obtain the subleading correction:
$$\lim_{n\to\infty} (F_{n,\ell}(\alpha)-\frac {\alpha^2} 2 \log n-\Osc_n(0,\alpha)+\Osc_n(0,0))=\alpha\tilde{e}_\ell+\frac {\alpha^2} 2 \tilde{v}_\ell-\log G(1+\alpha).$$

\end{rem}

We now interpret Corollary \ref{cor1} in terms of the Heine distribution. The following result is a direct generalization of Corollary \ref{fluct_reg2} for central disk droplets.

\begin{thm} Keep the notation of Corollary \ref{cor1} and
let $\{X_\nu^+,X_\nu^-\}$ be independent and Heine-distributed with $X_\nu^{\pm}\sim\He(\theta_{\nu,\alpha}^{\pm 1}\rho_\nu,\rho_\nu^2)$ where $\theta_{\nu,\alpha}$ is given in \eqref{rthdef}. Also let $W$ be a Gaussian with expectation $e_{h,\alpha}$ and variance $v_h$ and put
$K_n=\sum_{\nu=0}^{N-1}c_\nu\,x_\nu$ (again cf.~ \eqref{rthdef}). Then as $n\to\infty$,
$$F_{n,h}^{(\alpha)}(s)=sK_n+F_W(s)+\sum_{\nu=0}^{N-1}(F_{c_\nu X_\nu^+}(s)+F_{-c_\nu X_\nu^-}(s))+\bigO\Big(\frac{(\log n)^3}{n^\frac{1}{12}}\Big).
$$
\end{thm}

The proof is immediate by replacing $\mu_\nu(s)$ in the proof of Corollary \ref{fluct_reg2} by $s\mapsto \mu_\nu(s,\alpha)$ from \eqref{mudef}.

\subsection{Shallow outposts} We shall now investigate the situation when the droplet $S=A(a,b)$ is an annulus and the coincidence set $S^*$ contains a single outpost $\{|z|=t\}$
outside of $S$.
There are three examples to consider: $t>b$, $0<t<a$ and $t=0$.

\subsubsection{Case 1: $t>b$} \label{cgb}
Suppose that $Q$ is a potential such that the droplet $S=A(a,b)$ is an annulus (or disk) while the coincidence set is
$$S^*=S\cup\{z\,;\, |z|=t\}.$$

It is convenient to introduce a ``localized'' potential $\Qloc$ which equals $Q$ on a disc $D=\{|z|\le b'\}$ containing $S$ in its interior,
and satisfying $\Qloc=+\infty$ off $D$. By choosing $b'<t$ we ensure that $S^*[\Qloc]=S[\Qloc]=S$.

Given a suitable  radially symmetric test-function $h$,  we consider the partition function $\tilde{Z}_{n,sh}$ in potential
$$\tilde{Q}_{\rm{loc}}:=\Qloc-\frac {sh} n.$$
Since the potential $\Qloc$ has no shallow points, the large $n$ expansion of $\log \tilde{Z}_{n,sh}$
is given by the case $N=0$ of Theorem \ref{regular_expansion}. (We assume $\alpha=0$ to keep it simple.)

Now write $Z_{n,sh}$ for the partition function in potential $Q-\frac {sh} n$ and observe that
$$\log Z_{n,sh}=\log \tilde{Z}_{n,sh}+L_{n,h}(s),$$
where we write
\begin{equation}\label{lnh1}L_{n,h}(s):=\log \frac{Z_{n,sh}}{\tilde{Z}_{n,sh}}.\end{equation}

Similarly we denote by $F_{n,h}(s)$ and $\tilde{F}_{n,h}(s)$ the CGF for $\fluct_n h$ with respect to the potential $Q$ and $\Qloc$, respectively and we have
\begin{equation}\label{lnh2}F_{n,h}(s)-\tilde{F}_{n,h}(s)=L_{n,h}(s)-L_{n,0}(s).
\end{equation}

We have the following theorem.

\begin{thm} \label{shallow_expansion} Write
\begin{equation}\label{params}\rho:= \frac bt,\qquad \theta:=\sqrt{\frac{\Delta Q(b)}{\Delta Q(t)}},\qquad c:=h(t) -h(b),\qquad \mu(s):= e^{cs}\theta .\end{equation} Then for $|s|\le \log n$, we have as $n\to\infty$
$$L_{n,h}(s) = \log[(-\mu(s)\rho\,;\,\rho^2)_\infty]  +\bigO\left(n^{-1}(1+|s|)\right).
	$$
\end{thm}

In order to draw some probabilistic consequences,
we let $X$ be Heine distributed with \begin{equation}\label{heino}X\sim \He(\theta\rho,\rho^2).\end{equation}

Using the identity \eqref{euler_id} as in the proof of Corollary \ref{fluct_reg2}, we deduce that the function
$$
 \log [(-\mu(s)\rho;\rho^2)_\infty]-\log[(-\theta\rho;\rho^2)_\infty]
$$
equals precisely the cumulant generating function
$F_{cX}(s):=\log\mathbb{E}(e^{scX}).$

Differentiating with respect to $s$ in \eqref{lnh2} and integrating from $0$ to $s$, recalling that
$$\tilde{F}_{n,h}(s)=se_h+\frac {s^2} 2 v_h+ \mathcal{E}(n)$$ by Corollary \ref{fluct_reg},
we obtain the following corollary of Theorem \ref{shallow_expansion}.

\begin{cor} \label{onesided_jump}
 In above notation, the cumulant generating function
$F_{n,h}(s)=\log \E_n(e^{s\fluct_n h})$
obeys the large $n$ asymptotic, uniformly for $|s|\le \log n$,
$$F_{n,h}(s)=se_h+\frac {s^2} 2 v_h+F_{cX}(s)
+\mathcal{E}(n),$$
where $\calE(n)$ satisfies the same estimates as in Theorem \ref{regular_expansion}.
\end{cor}

Thus the random variables $\fluct_n h$ converge in distribution to $cX+W$ where $X\sim \He(\theta\rho,\rho^2)$, $W\sim N(e_h,v_h)$ and $X,W$ are independent.

 In a sense, $X$ describes a ``unilateral displacement'' of particles from the droplet to the shallow circle $\{|z|=t\}$. (By contrast, the displacement between components of the droplet is ``bilateral''.)

Moreover, in contrast with Corollary \ref{fluct_reg2}, the Heine distribution of $X$ is independent of $n$. Picking $h$ which is zero in a neighbourhood of $S$ and one in a neighbourhood
of the outpost $\{|z|=t\}$, we can identify the random variable
$$N_n:=\fluct_n h=\sum_{j=1}^n h(z_j)$$ with the number of particles which are found in a vicinity of the shallow outpost.

\begin{cor} As $n\to\infty$, the random variables $N_n$ converge in distribution to $\He(\theta\rho,\rho^2)$ where $\theta$ is given in \eqref{params} and $\rho=b/t$.
\end{cor}

The following simple proposition sheds some further light on the Heine distribution.

The $q$-analogue of a number $x$ is defined by $[x]_q = \frac{q^x-1}{q-1}$.

\begin{prop} If $X\sim \He(\theta\rho, \rho^2)$ then
$$\mathbb{E} [X]_{\rho^2} = \frac{\theta{\rho}}{1+\theta{\rho}} \frac{1}{1-\rho^2},\qquad \text{and}\qquad \mathbb{E} X = \sum_{j=0}^{\infty} \frac{\theta \rho^{2j+1}}{1+\theta \rho^{2j+1}}.$$
\end{prop}

\begin{proof} By \eqref{pekka}, for $Y\sim \He(\theta,q)$ and $s\in \R$ we have
\begin{align}\label{lol3}
\E(e^{sY})=	\sum\limits_{k=0}^{\infty} e^{sk} \cdot \mathbb{P}(\{Y=k\}) = \sum\limits_{k=0}^{\infty} e^{sk} \frac{1}{(-\theta;q)_\infty} \frac{q^{\frac{1}{2}k(k-1)} \theta^k}{(q;q)_k}
= \frac{(-\theta e^{s};q)_\infty}{(-\theta;q)_\infty}.
\end{align}
Taking first the derivative with respect to $s$, and then evaluating at $s=0$, we find (by ``partial fractions'')
$$\E Y = \sum_{j=0}^{\infty} \frac{\theta q^{j}}{1+\theta q^{j}},$$
and the claim about $\mathbb{E} X$ follows. Similarly,
$$\mathbb{E}(q^{Y}) =\frac 1 {(-\theta;q)_\infty}\sum_{k=0}^\infty q^k\frac {q^{\frac 1 2 k(k-1)}\theta^k}{(q;q)_k}= \frac{(-\theta q; q)_\infty}{(-\theta;q)_\infty} = \frac{1}{1+\theta},$$
and the claim about $\mathbb{E} [X]_{\rho^2}$ easily follows.
\end{proof}

\subsubsection{Case 2: $t<a$} Suppose that $S=A(a,b)$ where $a>0$
and $S^*=S\cup\{|z|=t\}$ where we begin by assuming $0<t<a$. This time we localize by
picking $a'$ with $t<a'<a$ and letting $\Qloc(z)=Q(z)$ for $|z|\ge a'$ and $\Qloc(z)=+\infty$ for $|z|<a'$.

Using similar definitions as above with respect to the localized partition function $\tilde{Z}_{n,sh}$ and the CGF $\tilde{F}_{n,h}(s)$, we write $L_{n,h}(s)$ for the function defined by
\eqref{lnh1}.

By adapting the reasoning above, we find that with
$$\rho:=\frac t a,\qquad c:=h(a)-h(t),\qquad \theta:=\sqrt{\frac{\Delta Q(t)}{\Delta Q(a)}},\qquad \mu(s):=e^{cs}\theta,$$
we have
\begin{equation}\label{chock}
L_{n,h}(s)= \log[(-\mu(s)\rho;\rho^2)_\infty]  +\bigO\left(n^{-1}(\log n)^2\right),
\end{equation}
and $L_{n,h}(s)-L_{n,h}(0)=F_{cX}(s)+\bigO\big(\frac{(\log n)^2}{n}\big)$ where $X\sim\He(\theta\rho,\rho^2)$.
(Furthermore, \eqref{chock} holds uniformly for $t \in [0,a-\epsilon]$ for any fixed $\epsilon >0$.)

Now let $t\to 0$ while keeping $a$ and $b$ fixed. In view of \eqref{chock}, this limit
leads to the convergence $L_{n,h}(s)\to 0$. We obtain the following result.

\begin{thm} If $S^*=S\cup\{0\}$ where $S=A(a,b)$, $0<a<b$, then $Z_{n,sh}=\tilde{Z}_{n,sh}+\bigO\big(\frac{(\log n)^2}{n}\big)$ as $n\to\infty$, where the $\bigO$-constant is uniform for $|s|\le \log n$.
Consequently,
the cumulant generating function $F_{n,h}(s)$ satisfies the Gaussian convergence $F_{n,h}(s)=se_h+\frac {s^2} 2v_h+\bigO\big(\frac{(\log n)^2}{n}\big)$ as $n\to\infty$.
\end{thm}

\begin{rem} In the general case of non-symmetric potentials, the Gaussian fluctuation theorem in \cite{AHM} says that if $S=S^*$ is connected with a smooth boundary, then
the cumulant generating function $F_{n,h}(s)$ of $\fluct_n h$, where $h$ is a smooth test-function, satisfies the Gaussian convergence
\begin{equation}\label{gaussc}F_{n,h}(s)=se_h+\frac {s^2} 2v_h+o(1).\end{equation}

It is natural to ask for more general conditions implying \eqref{gaussc}. We conjecture that, provided $Q$ is smooth and strictly subharmonic in a neighbourhood of $S^*$, the convergence \eqref{gaussc} should hold if and only if the shallow set $S^*\setminus S$ has capacity zero.
\end{rem}

\begin{rem} In the situation with several outposts, or outposts squeezed between different components of the droplet, our present methods can be used to work out asymptotics of $\log Z_{n,sh}$, but the probabilistic picture gets more involved, since several displacements must be considered simultaneously. For reasons of length, a probabilistic interpretation in this generality will not be carried out here.
\end{rem}

\subsection{Comments and related work}  \label{fuco}
Free energy asymptotics in the regular case $s=\alpha=0$ has been studied also for $\beta$-ensembles, for fairly general smooth potentials. It is by now known
that under suitable assumptions on $Q$, the free energy $\log Z_n^\beta$ in \eqref{pf1} obeys
\begin{equation}\label{betaexp}\log Z_{n}^\beta=-\frac \beta 2 n^2I_Q[\sigma]+\frac \beta 4 n\log n+n\left[C_\beta-(1-\frac \beta 4)E_Q[\sigma]\right]+o(n),\end{equation} where the constant $C_\beta$ is unknown when $\beta\ne 2$.
See the introduction of \cite{BKS} as well as \cite{BBNY,Se} and references therein.

The expansion beyond \eqref{betaexp} is still not settled as far as we know, but for example the sources \cite{ZW,CFTW} contain partial results and conjectures.
(Among other things, a term in the expansion of $\log Z_n^\beta$ proportional to $\sqrt{n}\log\frac \beta 2$ has been conjectured.)
Again we refer to the introduction in \cite{BKS} as  well as \cite{BBNY,Se} and the references therein as sources for these developments.

For $\beta=2$, a complete large $n$ expansion of $\log Z_n$ is known in some particular cases, e.g.~for the elliptic Ginibre ensemble, where asymptotics for the Barnes $G$-function provides the answer. In all known examples, it has been verified that the constant $C_2$ in \eqref{betaexp} is equal to $\frac 1 2 \log (2\pi)$, but the lower order terms (hidden in the ``$o(n)$'') remain an open problem also when $\beta=2$.

The appearance of a spectral determinant (corresponding to the term $F_Q[\sigma]$ in \eqref{lin2}) is predicted ``up to constants'' in \cite{ZW} for a large class of potentials giving rise to connected droplets.
This prediction was rigorously verified in \cite{BKS} for potentials that are smooth, strictly subharmonic and radially symmetric on $\C\setminus\{0\}$, and in \cite{ByunSeoYang} for the Ginibre ensemble with a strong point charge. In the case of disconnected droplets the term is new as far as we know.

The paper \cite{BKS} also gives a free energy expansion for Pfaffian Coulomb gas ensembles, see also the review \cite{BF}. It is possible and interesting to adapt our current analysis to this setting; this will be carried out in a separate publication \cite{Cr}.

The partition function has also been studied beyond the ``soft edge'' where $Q$ is smooth in a two-dimensional neighbourhood of the droplet. For example, free energy asymptotics with respect to various hard wall constraints have been investigated by many authors, see e.g. \cite{ForresterHoleProba, APS2009, C, C2, ByunPark2024} and the references therein. In these works, the
``hard wall'' is obtained by redefining a suitable potential $Q$ to be $+\infty$ outside of some prescribed closed set $C$.
The imposition of a hard wall might have a drastic effect for the Coulomb gas and might affect the structure of the equilibrium measure, see e.g. \cite{AR2017, ACCL1, C2}.

Free energy asymptotics with respect to hard edge potentials of the form $Q(z)=|z|^{2b}+\frac{\alpha}{n}\ell(z)+\infty\1_{\C\setminus C}(z)$ is worked out in \cite{C} in the case where $C$ is a union of annuli centered at the origin, and an oscillatory term of order $\bigO(1)$ in the expansion of $\log Z_n$ is found in terms of the Jacobi theta function \eqref{Jacobi theta}.


In the recent work \cite{JV}, the special case $Q=0$ is considered, with a hard wall along the boundary of a prescribed (simply connected) droplet $S$ with piecewise-analytic boundary. The authors express the $\bigO(1)$-term in the large $n$ expansion of the free energy in terms of
the Loewner energy of the hard wall $\d S$, and in terms of the Grunsky operator of $S$.
Underlying potentials of this kind are natural from the perspective of classical potential theory and appear also in the theory of truncated unitary random matrices, see \cite{ACM} and the references there. The situation is drastically different from hard wall ensembles in positive background as in \cite{ForresterHoleProba, APS2009,ACCL,C, Berezin} and the references therein.

The $q$-binomial formula \eqref{qbi} is well known in the theory of the Rogers-Ramanujan function, see e.g.~\cite{AAR,GR};
 its relationship to the Heine distribution can be gathered from e.g.~the sources \cite{FK,MH,Ke2,Ke}. In the Coulomb gas literature, displacements (or ``jumps'') of particles between the different components of $S$ are frequently described in terms of the Riemann theta function and the discrete normal distribution, both in soft and hard edge cases. (See \cite{C, ACC, ACCL} for results in dimension two, and e.g. \cite{DIZ1997, ClaeysGravaMcLaughlin, BG,CFWW,FahsKraII} and the references therein for results in dimension one.)

We next mention a few possibilities which are not included in our above analysis of rotationally symmetric potentials $Q$.
One such possibility is that the equilibrium density $\Delta Q$ vanishes to some order at the origin. In the paper \cite{AKS2} it is shown that this affects the limiting variance of the random variables \eqref{logf} in an interesting way (when the droplet is connected).
 Another possibility appears in the forthcoming work \cite{LR}, namely that $\Delta Q=0$ along some circle inside the droplet. (In this connection it would also potentially be interesting to investigate the case when
 $\Delta Q$ vanishes, for example, along a shallow outpost.) Yet another possibility is to allow $Q$ to have finite-sized jump discontinuities along some circle. Such discontinuities are different than those produced by a hard wall and are relevant in the study of counting statistics. See e.g. \cite{C2021 FH, BC FH 2022, ABES2023} and the references therein for studies in this direction. Other types of boundary conditions, besides soft and hard edge, are found in  \cite{AKS,Seo2021}, for example.

 The theory of Hermitian random matrices corresponds to $C=\R$, so that the gas is confined to the real axis, and the droplet
$S$ is a compact subset of $\R$. The case in which the droplet consists of several disjoint intervals is known as the multi-cut regime, and related partition functions are studied in \cite{ClaeysGravaMcLaughlin, BG, CFWW} and the references therein. In dimension one, the emergence of shallow points is known under the names ``birth of a cut'' or ``colonization of an outpost''. This was studied in \cite{A97,Eynard, Cl, Mo, BL} when the new cut has two soft edges, and in \cite{FahsKraI} when the new cut has two hard edges.

The emergence of the Euler characteristic in Theorem \ref{regular_expansion} is consistent with results in dimension one (see e.g. \cite{CFWW}) as well as in some two-dimensional cases (see e.g. \cite{BKS,BF,DS}), lending support to the hypothesis of an underlying universal result.

We finally note a certain formal resemblance between our present results to quantum tunneling for double potential wells in strong magnetic fields, which is discussed in the recent works \cite{FSW,HK}. (This is different from a Coulomb gas, and the analysis uses quite different methods.)

\subsection{Plan of this paper}
In Section \ref{genp}, we provide some general preliminaries and background.

In Section \ref{Sec_Reg}, we prove Theorem \ref{regular_expansion} in the case with a central annulus, i.e., $a_0>0$.

In Section \ref{reg2}, we adapt the proof in the annular case so as to obtain a proof of the expansion with a central disk ($a_0=0$), including the case of a Fisher-Hartwig singularity,
thus fully proving Theorem \ref{regular_expansion} and Theorem \ref{conical_expansion}.

In Section \ref{Sec_shallow}, we consider shallow outposts and prove Theorem \ref{shallow_expansion}.

\subsubsection*{Acknowledgement}
CC acknowledges support from the Swedish Research Council, Grant No.~2021-04626.

\section{Preliminaries} \label{genp}
In this section, we provide some background on the relationship between the partition function and weighted polynomials. After that we discuss a formula for the energy $I_Q[\sigma]$ and some useful integration techniques.
We finish with a detailed discussion of ``peak sets'' for weighted polynomials (of importance for our application of the Laplace method).

\subsection{The product formula for $Z_n$} Consider the  $L^2$-space over $\C$ with norm
$
\|f\|^2:=\int_\C |f|^{\,2}\, dA.$

The monic weighted orthogonal polynomial of degree $j$ in potential \eqref{pert} is denoted
$$
p_j(z)=p_{j,n}(z):=z^je^{-\frac 1 2 n\tilde{Q}(z)}.$$
(This simple formula depends, of course, on the radial symmetry of $\tilde{Q}$.)

By Andr\'{e}ief's identity (see \cite{Fo,M}) we have the
basic product formula
\begin{equation}\label{basis}
Z_n=n!\prod_{j=0}^{n-1}h_j,
\end{equation}
where $h_j$ is the squared norm  of $p_j$, i.e.,
\begin{equation*}
h_j=h_{j,n}[\tilde{Q}]:=\|p_j\|^2.
\end{equation*}

Our goal is to study the large $n$ asymptotics of $\log Z_n$, which, by \eqref{basis} and Stirling's formula, amounts to studying the sum
$\sum_{j=0}^{n-1}\log h_j$.

In what follows, we write $q$ for the function $Q$ restricted to $[0,\infty)$. Thus
\begin{equation}\label{w-op}
p_j(z)=z^j e^{\frac 12 sh(r)+\alpha\log r}e^{-\frac 1 2 nq(r)},\qquad (r=|z|).
\end{equation}

\subsection{Weighted energy of the equilibrium measure} The following lemma gives a convenient expression for the weighted logarithmic energy \eqref{Q-energy} of the equilibrium measure $d\sigma=\Delta Q\, \1_S\, dA$.

Recall the assumption \eqref{droplet} that the droplet is a disjoint union $S=\cup_{\nu=0}^N S^\nu$ where $S^\nu=A(a_\nu,b_\nu)$.

\begin{lem} \label{W-en-lem} We have
\begin{equation*}
I_Q[\sigma] = \int_{S} Q \, d\sigma +\sum_{\nu=0}^{N} \Bigg( \frac{1}{4}\int_{a_{\nu}}^{b_{\nu}} r q'(r)^2\, dr +  M_{\nu-1}^2\log a_{\nu}-M_{\nu}^2\log b_{\nu} \Bigg)
\end{equation*}
where the $M_\nu$'s are defined in \eqref{mass} and we use the convention $M_{-1}^2\log a_0=0\log 0:=0$ if  $a_0=0$.
\end{lem}

\begin{proof}
 Using the formula (see e.g. \cite[Example 5.7 of Section 0]{ST})
\begin{align}\label{useful integral}
\frac{1}{2\pi}\int_{0}^{2\pi}\log \frac{1}{|z-re^{i\theta}|}d\theta = \begin{cases}
\log \frac{1}{r}, & \mbox{if } r\geq |z|, \\
\log \frac{1}{|z|}, & \mbox{if } r<|z|,
\end{cases} \qquad r>0,
\end{align}
 we have for each $\nu$ that, with $\psi(r):=2r\Delta Q(r)$,
\begin{align*}
 \int_{S^\nu} \int_{S^\nu} \log \frac{1}{|z-w|}\, d\sigma(z)\,d\sigma(w)
& = \int_{a_{\nu}}^{b_{\nu}}  dr_z\int_{a_{\nu}}^{r_{z}} \psi(r_{z}) \psi(r_{w}) \log \frac{1}{r_{z}} \; dr_{w}\\
&\qquad + \int_{a_{\nu}}^{b_{\nu}}  dr_z\int_{r_{z}}^{b_{\nu}}  \psi(r_{z}) \psi(r_{w}) \log \frac{1}{r_{w}} \; dr_{w}.
\end{align*}

Since $2\psi=(rq')'$ the right hand side is evaluated to
\begin{align*}
\frac 14\int_{a_{\nu}}^{b_{\nu}} \Big(q(b_{\nu})-q(r)-a_{\nu}q'(a_{\nu})\log(\frac{1}{r})+b_{\nu}\log(\frac{1}{b_{\nu}})q'(b_{\nu})\Big)(rq''(r)+q'(r))
\,dr.
\end{align*}

Using that $M_{\nu} = \frac{b_{\nu}q'(b_{\nu})}{2} = \frac{a_{\nu+1}q'(a_{\nu+1})}{2}$ (see for example Lemma \ref{lamm} below) and integrating by parts it is straightforward to verify that
\begin{align*}
\int_{S^\nu} \int_{S^\nu} \log \frac{1}{|z-w|}\, d\sigma(z)\,d\sigma(w) & = \frac{1}{4} \int_{a_{\nu}}^{b_{\nu}} rq'(r)^2\, dr  +   M_{\nu-1}^2\log a_{\nu} - M_{\nu}^2\log b_{\nu} \\
&  + 2M_{\nu-1} (M_{\nu}\log b_{\nu} -M_{\nu-1} \log a_{\nu}) -  M_{\nu-1} (q(b_{\nu})-q(a_{\nu})).
\end{align*}

Also for $\nu < \eta$ we similarly obtain
\begin{align*}
\int_{S^\nu} &\int_{S^\eta} \log \frac{1}{|z-w|}\, d\sigma(z)\,d\sigma(w) = \frac{1}{2}(M_{\nu}-M_{\nu-1}) \Big( 2M_{\eta-1}\log a_{\eta}- 2M_{\eta}\log b_{\eta} + q(b_{\eta})-q(a_{\eta}) \Big).
\end{align*}

Summation over all relevant $\nu$ and $\eta$ completes the proof of the lemma.
\end{proof}

\subsection{The Euler-Maclaurin formula} In what follows, we often use the following
summation rule, which is found e.g.~in \cite[Section 2.10]{NIST} and \cite{BKS}.

\begin{thm}\label{em_formula} If $h$ is $2d$ times continuously differentiable on the interval $[m,n]$ where $n,m$ are positive integers, then
\begin{align*}
\sum_{j=m}^{n-1} h(j) = &\int_{m}^{n} h(x)\, dx - \frac{h(n)-h(m)}{2} + \sum_{k=1}^{d-1} \frac{B_{2k}}{(2k)!} \Big( h^{(2k-1)}(n)-h^{(2k-1)}(m)   \Big) + \epsilon_d.
\end{align*}
Here $B_{2k}$ is the $2k$-th Bernoulli number (in particular $B_2=\frac 1 6$)
and
\begin{align*}
|\epsilon_d| \leq \frac{4\zeta(2d)}{(2\pi)^{2d}} \int_{m}^n |h^{(2d)}(x)|\, dx,
\end{align*}
where $\zeta(s)$ is the Riemann zeta function.
\end{thm}

\subsection{The Laplace method} \label{laplace_me} We will frequently approximate integrals using the following version of Laplace's method.

\begin{lem} \label{laplem} Let $f(r)$ and $g(r)$ be smooth ($C^4$ and $C^6$-smooth, respectively, will do)
functions for $r>0$ and consider the integral
\begin{equation}\label{funcn}I(n)=\int_0^\infty f(r)e^{-ng(r)}\, dr.\end{equation}

Suppose that $g(r)$ attains a local minimum at $r=r_0$, and that $f(r)$ is supported in a neighbourhood of $r=r_0$ which contains no other local minima of $g(r)$.
Also suppose $g''(r_0)> 0$ and $f(r_{0})\neq 0$. Then, as $n\to+\infty$,
\begin{align*}
I(n) = \sqrt{\frac{2\pi}{g''(r_0)n}} e^{-ng(r_0)} f(r_0) \big(1+\frac{a}{n} +\bigO(n^{-2})\big),
\end{align*}
where
\begin{align}
\label{ADdef}  a=-\frac 1 8 \frac {d_4}{d_2^2}+\frac 5 {24}\frac {d_3^2}{d_2^3} + \frac{1}{2} \frac{f''(r_{0})}{f(r_{0})}\frac 1 {d_2} - \frac{1}{2} \frac{f'(r_{0})}{f(r_{0})}\frac {d_3}{d_2^2},\qquad
d_\ell=g^{(\ell)}(r_0).
\end{align}
\end{lem}

\begin{proof} The special case $f(r)=r$ is worked out in \cite{BKS}, and the general case can be proved similarly, using Taylor's formula and the moments of the Gaussian distribution.
(Also, the special case $f(r)=1$ is given in the remark following
\cite[Theorem 15.2.5]{S}, for example.) We omit details.
\end{proof}

Now fix a number $\tau$ with $0\le \tau\le 1$ and put $g=g_\tau$ where
\begin{equation}\label{gtaudef}
 g_\tau(r):=q(r)-2\tau\log r.
\end{equation}

Let $I_\tau(n)$ denote the integral \eqref{funcn} with $g=g_\tau$ and $f(r)=2r^{1+2\alpha}e^{sh(r)}$, i.e.,
\begin{equation}\label{itaun}
I_\tau(n)=2\int_0^\infty r^{1+2\alpha}e^{sh(r)}e^{-ng_\tau(r)}\, dr.
\end{equation}

The definition is chosen so that $I_\tau(n)=\|p_j\|^2$ if $\tau=j/n$, see \eqref{w-op}.

If $r=r_\tau$ is a solution to $g_\tau'(r)=0$, and if $Q$ is smooth at $r$, then (since $4\Delta Q=q''+r^{-1} q'$)
\begin{equation}\label{dtwo}
r q'(r) = 2\tau, \qquad g_\tau''(r)=4\Delta Q(r).
\end{equation}

The solutions $r=r_\tau$, which give local minima for $g_\tau$, will be called \textit{local peak points}.

Recall that $S^*$ denotes the coincidence set. For asymptotic purposes, it suffices to study solutions of \eqref{dtwo} in some (relatively) open neighbourhood $\calN$ in $[0,\infty)$ of the set $S^*\cap[0,\infty)$. We take $\calN$ small enough so
that $Q$
is $C^6$-smooth and strictly subharmonic on the set $\{ z=re^{i\theta}: r\in\overline{\calN} , 0\leq \theta<2\pi\}$.

The totality of local peak points in $\calN$ is denoted
\begin{equation}\label{locpeak}
\LocPeak(\tau)=\{r\in\calN\, ;\, g'_{\tau}(r)=0\}.
\end{equation}

Note that all points in $\LocPeak(\tau)$ are strict local minima, since $g_\tau''(r_\tau)=4\Delta Q(r_\tau)>0$.
As a consequence  there is at most one local peak point $r_{\tau}$ in the vicinity of a given connected component $C$ of $S^*\cap [0,\infty)$. Furthermore, this point depends smoothly on $\tau$.

Differentiating with respect to $\tau$ in $r_\tau q'(r_\tau)=2\tau$, we deduce the following lemma.

\begin{lem}\label{evolve} Let $C$ be a given connected component of $S^*\cap [0,\infty)$ and let $\mathcal{U}$ be a neighbourhood of $C$. If $\mathcal{U}$ is small enough, then
a solution to $rq'(r)=2\tau$ such that $r\in \mathcal{U}$ is unique. We denote it $r_{C,\tau}$ and it obeys the differential equation
\begin{align*}
\frac {dr_{C,\tau}}{d\tau}=\frac 1 {2r_{C,\tau} \Delta Q(r_{C,\tau})}.
\end{align*}
\end{lem}

We have the following main approximation lemma.

\begin{lem} \label{explicit1} For each connected component $C$ of $S^*\cap [0,\infty)$ such that there is a solution to $g'_\tau(r)=0$ with $r$ near $C$, we write $r=r_{C,\tau}$ for that solution.
 Then as $n\to\infty$ (with $f(r)=2r^{1+2\alpha}e^{sh(r)}$)
\begin{equation*}
I_\tau(n)=\sum_{C}\sqrt{\frac{2\pi} n}\frac {1}{\sqrt{4\Delta Q(r_{C,\tau})}} e^{-ng_\tau(r_{C,\tau})}f(r_{C,\tau})  \Big(1+\frac {a_{C,\tau}}n+\bigO(n^{-2})\Big),
\end{equation*}
where the sum extends over all such components and the error term is uniform in $\tau\in [0,1]$.

Here $a_{C,\tau}$ is given by \eqref{ADdef} with $g=g_\tau$, $r_0=r_{C,\tau}$. The first few coefficients $d_\ell$ are:
\begin{equation}\label{d234}\begin{cases}
d_{2} &= 4\Delta Q(r_{C,\tau}),\cr
d_3 &= 4 \partial_r \Delta Q(r_{C,\tau}) -\frac{4}{r_{C,\tau}} \Delta Q(r_{C,\tau}),\cr
d_4 &= 4 \partial_r^2 \Delta Q(r_{C,\tau}) + \frac{12}{r_{C,\tau}^2} \Delta Q(r_{C,\tau}) - \frac{4}{r_{C,\tau}} \partial_r \Delta Q(r_{C,\tau}).\cr
\end{cases}
\end{equation}
\end{lem}

\begin{proof}

By a standard estimate in e.g.~\cite[Lemma 2.1]{A}, we have, as $n\to\infty$
$$I_\tau(n)=(1+\bigO(e^{-cn}))\int_\calN f(r)e^{-ng_\tau(r)}\, dr.$$

Now use a partition of unity to write $f=\sum_C f_C+f_0$ where $f_C$ is supported in a small neighbourhood of component $C$ and $f_0$ vanishes on a neighbourhood of $S^*\cap[0,\infty)$; all functions are smooth.
As $n\to\infty$, we have  $I_\tau(n)=(1+\bigO(e^{-c'n}))\sum_C I_{C,\tau}(n)$ where $c'>0$ and $I_{C,\tau}(n)$ is defined as in \eqref{itaun} with $f$ replaced by $f_C$. The computations of the coefficients $d_3,d_4$ corresponding to $f_C$ are straightforward from Lemma \ref{laplem}, leading to the identities \eqref{d234}. The implied constant from Lemma \ref{laplem} depends smoothly on $\tau\in [0,1]$, which proves that the error term is indeed uniform.
\end{proof}

\subsection{Discussion of local peak points} \label{lapg} It is expedient to discuss the structure of those local peaks giving the essential contribution to the integral \eqref{itaun} for large $n$. This discussion requires a bit of book-keeping, but is otherwise quite elementary.

For each $\tau$, $0<\tau\le 1$, we let $S_\tau$ be the droplet and  $S_\tau^*$ the coincidence set with respect to the potential $Q/\tau$. Let us write $U_\tau$ for the unbounded component of $\C\setminus S_\tau$ and $\Gamma_\tau:=\{|z|=\beta_\tau\}$ for the boundary of $U_\tau$. The evolution $\tau\mapsto S_\tau$ is known as Laplacian growth.

It is important to note that $r=\beta_\tau$ solves the equation $rq'(r)=2\tau$, and thus $\beta_\tau$ is always a local peak point in the sense above.

To see this, we introduce the obstacle function $\check{Q}_\tau$ defined by the obstacle $Q$ but with growth $2\tau\log|z|+\bigO(1)$ near infinity. To be precise,
$\check{Q}_\tau(z)$ is the supremum of $s(z)$ where $s(w)$ is subharmonic on $\C$ with $s\le Q$ and $s(w)\le 2\tau\log|w|+\bigO(1)$ as $w\to\infty$.

 It is easy to see that $\check{Q}_\tau(z)$ is harmonic in $U_\tau$, where is takes the form
 \begin{equation}\label{har_obst}
 \check{Q}_\tau(z)=2\tau\log|z|+B_\tau,
 \end{equation}
 for some constant $B_\tau$. Moreover $\check{Q}_\tau(z)$ is globally $C^{1,1}$-smooth and coincides with $Q$ on $S_\tau$. Considering the derivative of $Q-\check{Q}_\tau$ in the radial direction, one verifies that $\beta_\tau q'(\beta_\tau)=2\tau$.

If we extend $\beta_\tau$ to $\tau=0$ by $\beta_0=a_0$, then $\tau\mapsto \beta_\tau$ is strictly increasing, continuous from the left, and satisfies
$\sigma(\D_{\beta_\tau})=\tau$.

The points of discontinuity of $\beta_\tau$ are precisely the points $\tau=M_\nu=\sigma(\D_{b_\nu})$ in \eqref{mass} for $0\le \nu\le N-1$; we have $\beta_{M_\nu}=b_\nu$
 while the right hand limit $\beta_{M_\nu}^+=\lim_{\tau \downarrow M_\nu}\beta_\tau$ equals to $a_{\nu+1}$.

In the following we denote by
$$V_\tau(z):=2\tau\log|z|+B_\tau$$
the harmonic continuation of $\check{Q}_\tau\big|_{U_\tau}$ to $\C\setminus\{0\}$.
Thus $(Q-V_\tau)(z)=g_\tau(|z|)-B_\tau$.

Note that, by the maximum principle, we have the inequality $V_{\tau} \leq \check{Q}_\tau$.

We now define a set $\Peak(\tau)$ of ``global peak points'' by
$$\Peak(\tau):=\{r\ge 0\,;\, g_\tau(r)=B_\tau\}.$$

Since $(Q-V_\tau)(\beta_\tau)=(Q-\check{Q}_\tau)(\beta_\tau)=0$, we always have that
$\beta_\tau\in \Peak(\tau)$. In a similar way, we see that, $S_\tau^* \cap[\beta_{\tau},\infty) \subset \Peak(\tau)$.

\begin{lem} \label{lamm} For each $\tau$ with $0< \tau\le 1$ we have
$\Peak(\tau)=S_\tau^* \cap[\beta_{\tau},\infty).$

In particular, if $0<\tau<1$ and if $\Peak(\tau)$ consists of more than one point, then $\tau=M_\nu$ for some $\nu$, $0\le\nu\le N-1$, and then
$\Peak(\tau)$ consists of the points $b_\nu,a_{\nu+1}$, and possibly finitely many shallow points $r$ in the gap $b_\nu<r<a_{\nu+1}$.

Also, $\Peak(1)$ consists of the point $\beta_1=b_N$ and all possible shallow points $r>b_N$; $\Peak(0)$ consists of the origin in the central disk case, and $a_0$ and possible shallow points $r$ with $0\le r<a_0$ in the annular case.
\end{lem}

\begin{proof} First assume $0<\tau<1$. Then $\beta_\tau>0$ and $g_\tau(\beta_\tau)=B_\tau$;
it remains only to prove that $\Peak(\tau)\subset S_\tau^* \cap[\beta_{\tau},\infty)$.

As before let $\calN$ be a small neighbourhood of $S^*\cap[0,\infty)$. If $r\not\in\calN$, we have $(Q-\check{Q}_\tau)(r)\ge c>0$ by \cite[Lemma 2.1]{A} (and its proof), so $\Peak(\tau)\subset \calN$.

 Since $V_\tau=\check{Q}_\tau$ on $[\beta_\tau,\infty)$ and $\check{Q}_\tau\le Q$ there with equality precisely on $S_\tau^* \cap[\beta_{\tau},\infty)$, it suffices to prove that there are no points $r\in\calN$ with $r<\beta_\tau$ at which $g_\tau(r)=B_\tau$.

However
if $r\in S_\tau^*$, $0<r<\beta_\tau$, and $g_\tau(r)=B_\tau$, then either $r=\beta_{\tau'}$ for some $\tau'<r$ or $r$ is a shallow point in a gap $b_\nu< r< a_{\nu+1}$ where $a_{\nu+1}<\beta_\tau$. In the first case
we have
$g_{\tau'}'(r)=0$, i.e., $rq'(r)=2\tau'<2\tau$, which shows that
\begin{equation}\label{contr}g_\tau'(r)=\frac {rq'(r)-2\tau} r=\frac {2(\tau'-\tau)} r< 0.\end{equation} In the second case we likewise have $g_{\tau'}'(r)=0$ with $\tau'=M_\nu$, so again \eqref{contr} holds.

Thus there are no solutions $r\in S_\tau^*$ with $r<\beta_\tau$
 to $g_\tau'(r)=0$. Moreover, \eqref{contr} shows that if the neighbourhood $\calN$ of $S^*\cap[0,\infty)$ is sufficiently small, there are no solutions $r<\beta_\tau$ in $\calN$ to the equation $g_\tau'(r)=0$.
The cases $\tau=0$ and $\tau=1$ are simple; we omit further details.
\end{proof}

Recall that $\LocPeak(\tau)$ denotes the set of local peak points in \eqref{locpeak}.

\begin{lem}\label{urban} With $\delta_\tau(r) = \dist(r,\LocPeak(\tau))$, we have the inequality (for all $r\ge 0$)
\begin{equation}\label{quadr}
g_{\tau}(r)-B_\tau \ge c \min\{\delta_\tau(r)^2,1\},
\end{equation}
where $B_\tau$ is defined in \eqref{har_obst} and $c>0$ is some constant independent of $\tau$, $0\le\tau\le 1$.
\end{lem}

\begin{proof} For points close to $\LocPeak(\tau)$, the relation \eqref{quadr} follows from Taylor's formula: if $r\in\LocPeak(\tau)$,
\begin{equation*}
g_{\tau}(r+h)-B_\tau =d + 4\Delta Q(r)\frac {h^2} 2+\bigO(h^3),\qquad (h\to 0),
\end{equation*}
where $d=d(r)\ge0$.

	This implies that there is $c_0>0$ and $c>0$ (depending on the infimum of $\Delta Q$ over $\calN$) such that if $\dist(r,\LocPeak(\tau))<c_0$ then $g_\tau(r)-B_\tau \ge c\delta_{\tau}(r)^2$.
	
	As we already noted
there is $c>0$ such that $(Q-\check{Q}_\tau)(r)\ge c$ for $r\not\in\calN$. Since $V_\tau\le Q$, we conclude the proof.
\end{proof}

While global peak points give the largest contribution, there may also be some local peak points $r$ which contribute significantly to the integral $I_\tau(n)$. In view of Lemma \ref{explicit1}, this will happen if the value $g_\tau(r)$ is close to the minimum $B_\tau$.

To be specific, we set \begin{equation}\label{dnen}
\delta_n=C\frac {\log n} n,\qquad \eps_n=\sqrt{\delta_n},
\end{equation}
where $C$ is a large constant, and define the set of \textit{significant local peak points} to be
\begin{align}\label{def of SLP}
\SignPeak(\tau)=\SignPeak(\tau,n) & :=
 \{r\in\LocPeak(\tau)\, ; \, g_\tau(r)< B_\tau+\delta_n\}.
\end{align}

In the following, we will say that a number $\tau\in [0,1]$ is a \textit{branching value} if the peak set $\Peak(\tau)$ consists of at least two points. Evidently, the values $M_0,M_1,\ldots, M_{N-1}$ are branching values,
and these are all in the open interval $(0,1)$. The value $\tau=0$ is a branching value if there is a shallow outpost $|z|=c$ with $c<a_0$ and $\tau=1$ is a branching value if there is an outpost with $c>b_N$.

Combining Lemmas \ref{evolve}--\ref{urban} we obtain the following result.

\begin{lem}\label{tenta}
 $\SignPeak(\tau)$ consists of a single point $r=r_\tau$ when $\tau$ is sufficiently far away from the branching values, in the sense that $|\tau-M_\nu|\ge c>0$ for all $\nu$.
If $\tau$ is close to $M_\nu$, $|\tau-M_\nu|<c$, there might be several significant local peaks (the end-points $b_\nu,a_{\nu+1}$ and possibly some shallow points in between, if $0\le \nu\le N-1$).
\end{lem}

We will later improve on the lemma; the following statement
elucidates more precisely the state of affairs.

\begin{prop} \label{reader} If $|\tau-M_\nu|\ge C(\log n)/n$ for all branching values $M_\nu$ where $C$ is large enough, then
$\SignPeak(\tau)$ consists of a single point in the interior of $S$.
\end{prop}

\begin{proof}[Remark on the proof] As we will not need the full generality, we do not give complete details.
We prove the special case without shallow outposts in Lemma \ref{onepk} below.
The key technical ideas are just the same in the case with one or several outposts, as is noted in Section \ref{Sec_shallow}. Using these remarks, it is not difficult to supply a complete proof.
\end{proof}

In what follows, we often refer to points of $\SignPeak(\tau)$ as the ``significant solutions to $g_\tau'(r)=0$''. When analyzing $I_\tau(n)$ (recall \eqref{itaun}), we must take care to include all significant solutions, but it will not matter if we include some additional ones from the set $\LocPeak(\tau)\setminus \SignPeak(\tau)$, since their contribution will anyway be negligible in comparison.

\begin{lem} \label{loc_lem}
For each $\tau$, $0\le\tau\le 1$, define
$$J_\tau=J_{\tau,n}:=\{r\ge 0\,;\, \dist(r,\SignPeak(\tau))<\eps_n\}$$
where $\eps_n$ is as in \eqref{dnen}. Also write
\begin{equation*}I_\tau^\sharp(n):= 2\int_{J_{\tau}}r^{1+2\alpha}e^{sh(r)}e^{-ng_{\tau}(r)}\, dr.\end{equation*}
Then if $C$ is large enough, the integral \eqref{itaun} satisfies
\begin{align*}
I_\tau(n)=I_\tau^\sharp(n) \cdot (1+\bigO(n^{-100})),
\end{align*}
where the error term is uniform for $0\le \tau\le 1$ and all real $s$ with $|s|\le \log n$.

\end{lem}

The proof is immediate from Lemma \ref{urban} and the definition \eqref{def of SLP} of $\SignPeak(\tau)$.

\section{The regular case: central annulus}\label{Sec_Reg}

In this section, we prove Theorem \ref{regular_expansion} under the assumption that
$a_0>0$, i.e., for an annular droplet.

The analysis in the central disk case $a_0=0$ is different, and we will treat the necessary modifications, including the case of a Fisher-Hartwig singularity, in Section \ref{reg2}.

\subsection{Setup in the annular case}
 Write $q(r)=Q(z)$ for $r=|z|$ and consider the perturbed potential
\begin{equation}\label{begin}
\tilde{Q}=Q-\frac {sh} n,
\end{equation}
where $h(z)=h(r)$ is a radially symmetric test-function, obeying the conditions in Subsection \ref{cge}. In the central annulus case we only consider $\alpha=0$; this is without loss of generality, since the function $\ell(z)=2\log |z|$ is smooth away from the origin and can therefore be absorbed into the definition of $h$.

As before
we write
$p_j(z)=z^je^{-n\tilde{Q}(z)/2}$ for the monic weighted orthogonal polynomials.

In the following we denote $\tau(j):= \frac{j}{n}$ and write $g_{\tau(j)}$ for the corresponding function in \eqref{gtaudef}. In this notation, the squared norm $h_j = \norm{p_j}^2$ is given by
\begin{equation}\label{hjint}
h_j=2\int_0^\infty r e^{sh(r)} e^{-ng_{\tau(j)}(r)}\, dr.
\end{equation}
(Note that $h_{j}$ is equal to $I_{\tau(j)}(n)$ with $\alpha=0$, see \eqref{itaun}.)

Our goal is to estimate the sum of logarithms: $\sum_{j=0}^{n-1} \log h_j$.

By Lemma \ref{loc_lem}, the main contribution
to the integral \eqref{hjint} comes from a small neighbourhood of the significant solutions $\SignPeak(\tau(j))$. Moreover, by Lemma \ref{tenta} the
set $\SignPeak(\tau(j))$ consists either of a single point in the interior of $S$, or of two points near the boundary of a gap $\{b_\nu<|z|<a_{\nu+1}\}$, depending on whether or not $\tau(j)$ is close to some
$M_\nu$ (recall that $S=S^{*}$ in this section).

To make this more precise, it is convenient to define \begin{equation}\label{lndeff}L_n=C\log n\end{equation} where $C$ is a large constant to be fixed later.

Recall that $M_\nu=\sigma(\{|z|\le b_\nu\})$, $m_\nu=\lfloor M_\nu n\rfloor$.

Assume that $0\le\nu\le N-1$. By Lemma \ref{evolve}, for $\tau$ in a neighbourhood of $M_\nu$ there are two continuous solutions
\begin{equation}\label{tsol}\tau\mapsto r_{\nu,\tau},\qquad \text{and}\qquad \tau\mapsto r_{\nu+1,\tau}\end{equation}
to $g_{\tau}'(r)=0$ with
$r_{\nu,M_\nu}=b_\nu$ and $r_{\nu+1,M_\nu}=a_{\nu+1}$.

We have the following lemma.

\begin{lem} \label{estimat} Suppose that $0\le \nu\le N-1$. As $\tau\to M_\nu$ we have that
\begin{equation}\label{Lip}
r_{\nu,\tau}=b_\nu\cdot (1+\bigO(\tau-M_\nu)),\qquad r_{\nu+1,\tau}=a_{\nu+1}\cdot (1+\bigO(\tau-M_\nu)).
\end{equation}
Also with $\rho_\nu=b_\nu/a_{\nu+1}$,
\begin{equation}\label{stock}
	g_{\tau}(r_{\nu,\tau})-g_{\tau}(r_{\nu+1,\tau}) = 2(M_\nu-\tau)\log \rho_\nu + \bigO((\tau-M_\nu)^2),
\end{equation}
and
\begin{equation}\label{stock2}
	\frac{g_{\tau}''(r_{\nu,\tau}) } {g_{\tau}''(r_{\nu+1,\tau})} = \frac{\Delta Q(b_\nu)}{\Delta Q(a_{\nu+1})}\cdot (1 + \bigO(\tau-M_\nu)).
\end{equation}
\end{lem}

\begin{proof}  By Lemma \ref{evolve}, the solutions \eqref{tsol} are well-defined and Lipschitz continuous for $\tau$ near $M_\nu$, proving \eqref{Lip}. The equation \eqref{stock} follows from \eqref{equi}  and \eqref{Lip} since $\tau\mapsto \Delta Q(r_{\mu,\tau})$ is Lipschitz for the relevant values of $\mu$ and $\tau$.
Since $g_{M_\nu}''(r_{\nu,M_\nu})=4\Delta Q(b_\nu)$ and $g_{M_\nu}''(r_{\nu+1,M_\nu})=4\Delta Q(a_{\nu+1})$, we also obtain \eqref{stock2}.
\end{proof}

The estimate \eqref{stock} implies the following, where we recall from \eqref{dnen} that $\delta_n=C(\log n)/n$. (\footnote{We adopt the convention of denoting by the same symbol ``$C$'' various unspecified constants whose exact values can change meaning from time to time.})

\begin{lem} \label{onepk}
If the constant $C$ in \eqref{lndeff} is large enough, then
for all $\tau$ in a small neighbourhood of $M_\nu$ such that
$L_n/n\le |\tau-M_\nu|\le c$ we have $|g_\tau(r_{\nu,\tau})-g_\tau(r_{\nu+1,\tau})|\ge \delta_n$. Consequently, for such $\tau$, $\SignPeak(\tau)$ consists of
one single point (either $r_{\nu,\tau}$ if $\tau<M_\nu$, or $r_{\nu+1,\tau}$ if $\tau>M_\nu$).
\end{lem}

\begin{proof} Recall that $B_\tau$ is the minimum of $g_\tau$. For $\tau$ with $|\tau-M_\nu|\le c$ the value $B_\tau$ is attained at either $r_{\nu,\tau}$ or $r_{\nu+1,\tau}$ if $c$ is small enough, see Lemma \ref{tenta}.
 The hypothesis implies that either $g_\tau(r_{\nu,\tau})\ge B_\tau+\delta_n$ or $g_\tau(r_{\nu+1,\tau})\ge B_\tau+\delta_n$.
 The values of $g_\tau$ at possible other local peak points are at least $B_\tau+c_1$ for some $c_1>0$.
\end{proof}

Let $j$ be an integer, $0\le j\le n-1$. We have shown that we have the following alternatives.

\smallskip
\noindent
Case 1. If $|j-m_\nu|\ge L_n$ for all $\nu$ with $0\le \nu\le N-1$ then there is only one significant solution to $g_{\tau(j)}'(r)=0$, located in $S$, and of distance at least $\eps_n$ from all boundary points of $S$, where $\eps_n$ is given in \eqref{dnen}.
If this $r_{\tau(j)}$ lies in $S^\nu$ we denote it by $r_{\nu,\tau(j)}$.

\smallskip
\noindent
Case 2. If there is $\nu$ with $0\le \nu\le N-1$ such that $|j-m_\nu|<L_n$, then there are at most two significant solutions $r_{\nu,\tau(j)}$ and $r_{\nu+1,\tau(j)}$
located near $r=b_\nu$ and $r=a_{\nu+1}$, respectively. (One of them might be insignificant, but we anyway include both in our analysis below.)

\smallskip
\noindent

It is important to note that the above conclusions also hold when $\tau(j)$ is replaced by a real parameter $\tau$ (and then
$h_j$ is replaced by $I_\tau(n)$ in \eqref{itaun}).

Before proceeding, we summarize: by Lemma \ref{evolve}, each of the (one or two) solutions to $g_\tau'(r)=0$ obeys the differential equation
\begin{equation}\label{difft}\frac {d r_{\mu,\tau}}{d\tau}=\frac 1 {2r_{\mu,\tau}\Delta Q(r_{\mu,\tau})}.
\end{equation}
Also, if $0\le\nu\le N-1$, we have the branching $r_{\nu,M_\nu}=b_\nu$ and $r_{\nu+1,M_\nu}=a_{\nu+1}$ while
\begin{equation}\label{equi}
g_{M_\nu}(r_{\nu,M_\nu}) = g_{M_\nu}(r_{\nu+1,M_\nu}),
\end{equation}
the common value being the
global minimum $B_{M_\nu}$ of $g_{M_\nu}(r)$.

\subsection{Basic approximation lemmas} Let $\eps_n$ be as in \eqref{dnen}.

For each $j$ and each significant solution $r_{\nu,\tau(j)}$, we write $J_{\nu,j}=\{r\ge 0\,;\, |r-r_{\nu,\tau(j)}|<\eps_n\}$ and
\begin{equation}\label{hnuj}
h_{\nu,j}:=2\int_{J_{\nu,j}}re^{sh(r)}e^{-ng_{\tau(j)}(r)}\, dr.
\end{equation}
Following \cite{BKS}, we introduce the function $\calB(r)$ for $r\in\calN$ by
\begin{equation}\label{br}
\calB(r):=-\frac 1 {32}\frac {\d_r^2\Delta Q(r)}{(\Delta Q(r))^2}-\frac {19} {96 r}\frac {\d_r\Delta Q(r)}{(\Delta Q(r))^2}+\frac 5 {96} \frac {(\d_r\Delta Q(r))^2}{(\Delta Q(r))^3}+\frac 1 {12r^2}\frac 1 {\Delta Q(r)}.
\end{equation}

By the Laplace method in Subsection \ref{laplace_me}, we have
\begin{equation}\label{norm}
h_{\nu,j}=\sqrt{\frac {2\pi} n} \frac {2r_{\nu,\tau(j)}} {\sqrt{g_{\tau(j)}''(r_{\nu,\tau(j)})}}  e^{sh(r_{\nu,\tau(j)})}e^{-ng_{\tau(j)}(r_{\nu,\tau(j)})} \Big(1+\frac {a_{\nu,j}} n+ \bigO(n^{-2}) \Big),
\end{equation}
as $n \to \infty$, where
\begin{equation}\label{anjk}
a_{\nu,j}=\calB(r_{\nu,\tau(j)})
+ \frac{s^2h'(r_{\nu,\tau(j)})^2 }{2} \frac{1}{d_{2}}  + \frac{sh''(r_{\nu,\tau(j)})}{2}\frac{1}{d_{2}} +\frac{sh'(r_{\nu,\tau(j)})}{r_{\nu,\tau(j)}}\frac{1}{d_{2}}  -\frac{sh'(r_{\nu,\tau(j)})}{2} \frac{d_{3}}{d_{2}^{2}}.
\end{equation}
In \eqref{anjk},
\begin{equation}\label{dlnk}
d_{\ell}=d_{\ell,\nu,j}:= g_{\tau(j)}^{(\ell)} (r_{\nu,\tau(j)}).
\end{equation}

Next write $h_j^\sharp$ for
\begin{equation}\label{hjsharp}
h_j^\sharp=\begin{cases}h_{\nu,j},\quad \text{or}\cr
h_{\nu,j}+h_{\nu+1,j},\cr\end{cases}
\end{equation}
depending on whether there are one or two significant solutions to $g_{\tau(j)}'(r)=0$.

By Lemma \ref{loc_lem}, the quantities $h_j^\sharp$ are good approximations to $h_j$ in the sense that for $|s|\le \log n$,
\begin{equation}\label{hjas}
h_j = h_j^\sharp\cdot \big(1+\bigO(n^{-100})\big).
\end{equation}

We turn to an estimate for $h_j^\sharp$ which is useful when $j$ is close to a critical index $m_\nu$.

Recall from \eqref{rthdef}, \eqref{mudef} that (since $\alpha=0$)
\begin{equation}\label{rdef2}\rho_\nu=\frac {b_\nu}{a_{\nu+1}},\qquad \mu_\nu(s):=\mu_\nu(s,0;n,h)=e^{s(h(a_{\nu+1})-h(b_\nu))}\sqrt{\frac {\Delta Q(b_\nu)}{\Delta Q(a_{\nu+1})}}\rho_\nu^{2x_\nu}.\end{equation}

\begin{lem}\label{qnormlem}  Let $0\le \nu\le N-1$. Then for $|j-m_\nu|\le L_n$, $|s|\leq C \log n$, with $m_\nu=\lfloor M_\nu n\rfloor$,
\begin{align*}
\frac{h_{\nu+1,j}}{h_{\nu,j}} & = \frac{\mu_\nu(s)}{\rho_\nu}  \rho_\nu^{2(m_\nu- j)} \big[1+ \bigO((1+|s|)(\tau(j)-M_\nu)) + \bigO(n(\tau(j)-M_\nu)^2)\big].
\end{align*}
\end{lem}

\begin{proof}
By \eqref{norm}, we have
\begin{equation*}
\begin{split}
\frac{h_{\nu+1,j}}{h_{\nu,j}} &= \frac{r_{\nu+1,\tau(j)}}{r_{\nu,\tau(j)}} \sqrt{\frac{g''_j(r_{\nu,\tau(j)})}{g''_j(r_{\nu+1,\tau(j)})}} e^{sh(r_{\nu+1,\tau(j)}) -sh(r_{\nu,\tau(j)}) } e^{n(g_{\tau(j)}(r_{\nu,\tau(j)})-g_{\tau(j)}(r_{\nu+1,\tau(j)}) )} \big(1 + \bigO(n^{-1})\big).\\
\end{split}
\end{equation*}
Inserting the estimates in Lemma \ref{estimat} in the above relation, we finish the proof.
\end{proof}

\subsection{Proof of Theorem \ref{regular_expansion} in the annular case $a_0>0$}

We shall estimate the sum $\sum_{j=0}^{n-1}\log h_j$, making use of the approximation $\log h_j^\sharp$ obtained by combining \eqref{norm} and \eqref{hjsharp}.
The form of the approximations prompts us to evaluate several sums, which is done in a series of lemmas.

Our strategy is to first group together terms $\log h_{\nu,j}$ in \eqref{norm} according to the ``blocks''
$m_{\nu-1}\le j<m_\nu$, and then to correct for the $j$ which fall near one of the critical indices $m_\nu$ with $0\le \nu\le N-1$.

More precisely, we observe that by \eqref{hjas}
\begin{align}\label{the text is hard to follow without this equation}
\sum_{j=0}^{n-1}\log h_j =\sum_{\nu=0}^N\sum_{j=m_{\nu-1}}^{m_\nu-1}\log h_{\nu,j} + \sum_{\nu=0}^{N-1} T_{\nu}+\bigO(n^{-99}),
\end{align}
where (with $L_n=C\log n$)
\begin{align}
T_\nu:=	\sum\limits_{j=m_{\nu}}^{m_{\nu}+L_n}  \log\Big( 1 &+ \frac{h_{\nu,j}}{h_{\nu+1,j}} \Big)  + \sum\limits_{j=m_{\nu}-L_n}^{m_{\nu}-1}  \log\Big( 1 + \frac{h_{\nu+1,j}}{h_{\nu,j}} \Big).\label{tndefn}
\end{align}

The terms $T_\nu$, which are estimated in Lemma \ref{theta_contribution} below, are closely connected to to the Heine distribution.

\smallskip

We extend the definition of $\tau$ to real $t$
$$\tau(t)=\frac t n$$
and
for fixed $\nu$ with $0\le \nu\le N$, we write
\begin{equation}\label{gnt}
\gamma_{\nu} (t) := g_{\tau(t)}(r_{\nu,\tau(t)}).
\end{equation}

We start with the following lemma, which generalizes \cite[Lemma 2.3]{BKS} from the case $N=0$.

\begin{lem} \label{l1} With $x_\nu=M_\nu n-m_\nu$, we have as $n\to\infty$,
\begin{align*}
\sum_{\nu=0}^{N} \sum_{j=m_{\nu-1}}^{m_{\nu}-1} -ng_{\tau(j)}(r_{\nu,\tau(j)}) &= -n^2 I_Q[\sigma] + \frac{n}{2} (q(b_N)-2\log b_N-q(a_0)) \\
& + \frac{1}{6}\log \frac {b_N}{a_0}+\sum_{\nu=0}^{N-1} \Big(x_\nu^2 + x_\nu+\frac{1}{6}\Big) \log \rho_\nu + \bigO(\frac{1}{n}).
\end{align*}
\end{lem}
\begin{proof}
By the Euler-Maclaurin formula (Theorem \ref{em_formula} with $d=2$) we have
\begin{equation}\label{em0}\begin{split}
\sum_{j=m_{\nu-1}}^{m_{\nu}-1} g_{\tau(j)}(r_{\nu,\tau(j)}) = \int_{m_{\nu-1}}^{m_{\nu}}& \gamma_{\nu}(t)\, dt
 -\frac{\gamma_{\nu}(m_{\nu})-\gamma_{\nu}(m_{\nu-1})}{2}
 + \frac{\gamma'_{\nu}(m_{\nu})-\gamma'_{\nu}(m_{\nu-1})}{12} + \epsilon.
 \end{split}
 \end{equation}
 The error term $\epsilon$ will be estimated shortly. Since $g_{\tau}'(r_{\nu,\tau})=0$ we have
$$\gamma_\nu'(t)=-\frac 2 n\log r_{\nu,\tau(t)}.$$
Recalling from \eqref{difft} that $\frac{dr_{\nu,\tau}}{d\tau} = \frac{1}{2r_{\nu,\tau} \Delta Q(r_{\nu,\tau})}$, we obtain
\begin{align*}
\gamma_\nu''(t)=-\frac 1 {n^2} \frac 1 {r_{\nu,\tau(t)}^2\Delta Q(r_{\nu,\tau(t)})}, \qquad \gamma_\nu^{(3)}(t)=\bigO(n^{-3}),\qquad \text{and}\qquad  \gamma_\nu^{(4)}(t)=\bigO(n^{-4}),
\end{align*}
where we used $r_{\nu,\tau(t)}\ge a_0>0$. Hence, in view of Theorem \ref{em_formula}, the error term $\epsilon$ in \eqref{em0} is $\bigO(n^{-3})$.

Next observe that the integral in \eqref{em0} equals to
\begin{align}\label{lol2}
\int_{m_{\nu-1}}^{m_{\nu}} \gamma_{\nu}(t)\, dt = \int_{M_{\nu-1}n}^{M_\nu n} \gamma_{\nu}(t)\, dt + \int_{m_{\nu-1}}^{M_{\nu-1}n} \gamma_{\nu}(t)\, dt -\int_{m_\nu }^{M_\nu n} \gamma_{\nu}(t)\, dt.
\end{align}
We now make the substitution
\begin{equation}\label{change_var}
u = r_{\nu,\tau(t)}, \qquad du = \frac{dt}{2 n u \Delta Q(u)}.
\end{equation}
This gives (using that $uq'(u)=2\tau(t)$)
	$$
	\frac{1}{n}\int_{M_{\nu-1}n}^{M_\nu n} \gamma_\nu(t)\, dt = 2   \int_{a_{\nu}}^{b_{\nu}}  (q(u) -uq'(u)\log u )u\Delta Q(u)\, du,
	$$
which  can be rewritten as
\begin{equation}\label{main}
\frac{1}{n}\int_{M_{\nu-1}n}^{M_\nu n} \gamma_\nu(t)\, dt = \int_{A(a_\nu,b_\nu)} Q\Delta Q\, dA + \frac{1}{4}\int_{a_\nu}^{b_\nu} uq'(u)^2\, du + M_{\nu-1}^2\log a_\nu-M_{\nu}^2\log b_\nu.
\end{equation}
	
By Lemma \ref{estimat}, we verify easily that (for $0\le \nu\le N-1$)
\begin{align*}
\gamma_{\nu+1}(t)-\gamma_{\nu}(t) = 2(\tau(t)-M_{\nu})\log \rho_{\nu} + \bigO((\tau(t)-M_\nu)^2),\qquad (\tau(t) \to M_\nu),
\end{align*}
which gives
\begin{equation}\label{rest}
\int_{m_\nu}^{M_\nu n} (\gamma_{\nu+1}(t)-\gamma_{\nu}(t))\, dt = \int_{m_\nu}^{M_\nu n } 2(\tau(t)-M_{\nu}) \log \rho_\nu\, dt +\bigO(\frac{1}{n^2}) = -\frac{x_\nu^2}{n}\log \rho_\nu + \bigO(\frac{1}{n^2}).
\end{equation}
Combining \eqref{lol2}, \eqref{main} and \eqref{rest} with Lemma \ref{W-en-lem} gives,
\begin{align}\label{1est}
\sum_{\nu=0}^{N} \int_{m_{\nu-1}}^{m_\nu} \gamma_\nu(t)\, dt = n I_Q[\sigma] - \frac{1}{n}\sum_{\nu=0}^{N-1} x_\nu^2\log \rho_\nu + \bigO(\frac{1}{n^2}).
\end{align}
There remains to estimate the last two terms in \eqref{em0}.
To this end, we use (recall $m_{-1}=0$, $m_N=n$)
\begin{align*}
g_0(m_{-1}) = q(a_0) \qquad \text{and} \qquad g_N(m_N) = q(b_N)-2\log b_N,
\end{align*}
to get
\begin{align*}
\sum\limits_{\nu=0}^{N} -\frac{1}{2} \Big( \gamma_\nu(m_\nu)-\gamma_\nu(m_{\nu-1})   \Big) = \frac{1}{2} \Big(q(a_0)-q(b_N)+2\log b_N  \Big) + \sum\limits_{v=0}^{N-1} \frac{1}{2} \Big(\gamma_{\nu+1}(m_\nu)-\gamma_{\nu}(m_\nu) \Big).
\end{align*}
Moreover, by \eqref{stock} with $\tau=\tau(m_\nu)$,
\begin{align*}
\gamma_{\nu+1}(m_{\nu})-\gamma_{\nu}(m_{\nu}) = \frac{2(m_{\nu}-nM_\nu)}{n}\log \rho_\nu + \bigO((\frac{m_\nu}{n}-M_\nu )^2) = -\frac{2x_\nu}{n}\log \rho_\nu +\bigO(\frac{1}{n^2}),
\end{align*}
whence
\begin{equation}\label{2est}\begin{split}
- \frac 1 2 \sum_{\nu=0}^{N} (\gamma_\nu(m_\nu)-\gamma_\nu(m_{\nu-1})) &= \frac{q(a_0)-q(b_N)+2\log b_N}{2} - \sum_{\nu=0}^{N-1} \frac{x_\nu}{n}\log \rho_\nu + \bigO(\frac{1}{n^2}).
\end{split}
\end{equation}
Recalling that $\gamma_\nu'(t)=-\frac 2 n \log r_{\nu,\tau(t)}$, we have (since $a_0>0$)
\begin{align*}
\gamma'_N(m_N) = -\frac{2}{n}\log b_N \qquad \text{and} \qquad \gamma'_0(m_{-1}) = -\frac{2}{n}\log a_0,
\end{align*}
and
\begin{align*}
\gamma'_{\nu+1}(m_{\nu}) - \gamma'_{\nu}(m_\nu) = \frac{2}{n} \log \frac {b_{\nu}}{a_{\nu+1}}+\bigO(\frac 1 {n^2})  = \frac{2}{n}\log\rho_\nu +\bigO(\frac 1 {n^2}).
\end{align*}
Summing up,
\begin{equation}\label{3est}
\sum_{\nu=0}^{N} \frac{\gamma'_\nu(m_\nu) -\gamma'_\nu(m_{\nu-1})}{12} = \frac{1}{6n}\log a_0 -\frac{1}{6n}\log b_N -\frac{1}{6n}\sum_{\nu=0}^{N-1}\log \rho_\nu+\bigO(\frac 1 {n^2}).
\end{equation}
	
Combining \eqref{em0}, \eqref{1est}, \eqref{2est}, \eqref{3est}, we conclude the proof of the lemma.
\end{proof}

We now turn to two lemmas, which together generalize \cite[Lemma 2.4]{BKS}.

\begin{lem} \label{l22} With $E_Q[\sigma]=\int_\C \log \Delta Q\, d\sigma$, we have
as $n\to\infty$
	\begin{align*}
	\sum\limits_{\nu=0}^N \sum\limits_{j=m_{\nu-1}}^{m_\nu-1}  \log \Delta Q(r_{\nu,\tau(j)}) = n E_Q[\sigma]  + \frac{1}{2} \log\frac{\Delta Q(a_0)}{\Delta Q(b_N)} - \sum\limits_{\nu=0}^{N-1} (x_\nu+\frac{1}{2})  \log \frac{\Delta Q(b_\nu)}{\Delta Q(a_{\nu+1})} + \bigO(\frac{1}{n}).
	\end{align*}
\end{lem}

\begin{proof}  We first use the Euler-Maclaurin formula (Theorem \ref{em_formula} with $d=1$) to write
\begin{equation}\label{em1}
\sum_{j=m_{\nu-1}}^{m_\nu-1} \log \Delta Q(r_{\nu,\tau(j)}) = \int_{m_{\nu-1}}^{m_\nu} \log \Delta Q(r_{\nu,\tau(t)})\, dt - \frac{1}{2}\log \frac{\Delta Q(r_{\nu,\tau(m_\nu)})}{\Delta Q(r_{\nu,\tau(m_{\nu-1})})}  + \bigO(\frac{1}{n}).
\end{equation}

The error term in \eqref{em1} comes from observing that $p_\nu(t):=\log \Delta Q(r_{\nu,\tau(t)})$ satisfies
\begin{align*}
p_\nu'(t) =
\frac 1 n \frac{\partial_r \Delta Q(r_{\nu,\tau(t)})}{\Delta Q(r_{\nu,\tau(t)})} \frac 1 {2r_{\nu,\tau(t)} \Delta Q(r_{\nu,\tau(t)})},\qquad \text{and}\qquad p_\nu''(t)=\bigO(n^{-2}),
\end{align*}
where we used \eqref{difft} and
$r_{\nu,\tau}\ge a_0>0$.

By the change of variables in $\eqref{change_var}$ we have, with $d\sigma=\Delta Q\,\1_S\, dA$,
\begin{align*}
\int_{M_{\nu-1}n}^{M_\nu n} \log \Delta Q(r_{\nu,\tau(t)})\, dt = n\int_{A(a_\nu,b_\nu)} \log \Delta Q\, d\sigma.
\end{align*}

In a similar way as in the proof of Lemma \ref{l1}, we deduce that
\begin{align*}
\int_{m_{\nu-1}}^{m_\nu}  \log \Delta Q(r_{\nu,\tau(t)})\, dt - \int_{nM_{\nu-1}}^{nM_\nu}   \log \Delta Q(r_{\nu,\tau(t)})\, dt =  x_{\nu-1} \log \Delta Q(a_\nu)  - x_\nu \log \Delta Q(b_\nu)+ \bigO(\frac{1}{n}).
\end{align*}

Combining the above estimates, we obtain
\begin{equation}\label{first2}
\sum_{\nu=0}^{N}\, \int_{m_{\nu-1}}^{m_\nu} \log \Delta Q(r_{\nu,\tau(t)})\, dt = n\int_{S} \log \Delta Q(z)\, d\sigma(z) -\sum_{\nu=0}^{N-1} x_\nu \log \frac{\Delta Q(b_\nu)}{\Delta Q(a_{\nu+1})} + \bigO(\frac{1}{n}).
\end{equation}

Finally, we note that
\begin{equation}\label{second2}
\sum_{\nu=0}^{N}  - \frac{1}{2}  \log \frac{\Delta Q(r_{\nu,\tau(m_\nu) })}{\Delta Q(r_{\nu,\tau(m_{\nu-1}) })} = \frac{1}{2} \log\frac{\Delta Q(a_0)}{\Delta Q(b_N)}
 -\frac{1}{2} \sum_{\nu=0}^{N-1} \log \frac{\Delta Q(b_\nu)}{\Delta Q(a_{\nu+1})} + \bigO(\frac{1}{n}).
\end{equation}	

Combining \eqref{em1} with \eqref{first2}, \eqref{second2}, we finish the proof of the lemma.
\end{proof}

\begin{lem} \label{l36} As $n\to\infty$, we have
\begin{align*}
\sum\limits_{\nu=0}^{N} \sum\limits_{j=m_{\nu-1}}^{m_\nu-1} \log r_{\nu,\tau(j)} = -\frac{n}{2}(q(b_N)-2\log b_N -q(a_0)) + \frac{1}{2}\log \frac{a_0}{b_N} - \sum\limits_{\nu=0}^{N-1} (x_\nu+\frac{1}{2}) \log \frac{b_\nu}{a_{\nu+1}} + \bigO(\frac{1}{n}).
\end{align*}
\end{lem}

\begin{proof} Using Theorem \ref{em_formula} with $d=1$, we get
\begin{equation}\label{app0}\begin{split}
\sum_{j=m_{\nu-1}}^{m_\nu -1} \log r_{\nu,\tau(j)} = \int_{m_{\nu-1}}^{m_\nu} &\log r_{\nu,\tau(t)}\, dt -\frac 1 2 \left(\log r_{\nu,\tau(m_\nu)} -\log r_{\nu,\tau(m_{\nu-1})}\right) + \bigO(\frac{1}{n}),
\end{split}
\end{equation}
where the error term is estimated in a similar way as in the proof of Lemma \ref{l22}.

We recall that $m_\nu=\lfloor M_\nu n\rfloor$ and use the change of variables \eqref{change_var} to deduce that
\begin{align}\label{expell1}
\int_{M_{\nu-1}n}^{M_{\nu}n} \log r_{\nu,\tau(t)}\, dt = n\int_{a_\nu}^{b_\nu} (\log s) 2s\Delta Q(s)\, ds.
\end{align}
A computation (using $4\Delta Q=q''+(1/s)q'$) gives
\begin{align*}\int\limits_{a_\nu}^{b_\nu} (\log s) 2s\Delta Q(s)\, ds&=M_\nu\log b_\nu-M_{\nu-1}\log a_\nu-\frac {q(b_\nu)-q(a_\nu)}{2} = \frac {g_{M_{\nu-1}}(r_{\nu,M_{\nu-1}})-g_{M_{\nu}}(r_{\nu,M_\nu})}{2}.
\end{align*}
A summation using that $g_{M_{\nu-1}}(r_{\nu,M_{\nu-1}})=g_{M_{\nu-1}}(r_{\nu-1,M_{\nu-1}})$,
$g_{1}(b_N)=q(b_N)-2\log b_N$ and $g_{0}(a_0)=q(a_0)$ now gives
\begin{align}\label{expell2}
\sum_{\nu=0}^{N} \int_{M_{\nu-1}n}^{M_{\nu}n} \log r_{\nu,\tau(t)}\, dt = -\frac{n}{2}(q(b_N)-2\log b_N -q(a_0)).
\end{align}
By similar computations we obtain
\begin{align*}
\int_{m_{\nu-1}}^{m_\nu} \log r_{\nu,\tau(t)}\, dt -	\int_{nM_{\nu-1}}^{nM_{\nu}} \log r_{\nu,\tau(t)}\, dt = -x_\nu \log b_\nu + x_{\nu-1}\log a_\nu + \bigO(\frac{1}{n}).
\end{align*}	
Combining the above, we conclude that
\begin{equation}\label{third1}\begin{split}
\sum_{\nu=0}^{N}\int_{m_{\nu-1}}^{m_\nu} \log r_{\nu,\tau(t)}\, dt =&  -\frac{n}{2}(q(b_N)-2\log b_N -q(a_0)) -\sum_{\nu=0}^{N-1} x_\nu \log \frac{b_\nu}{a_{\nu+1}} + \bigO(\frac{1}{n}).
\end{split}
\end{equation}
Finally, using that $|\tau(m_\nu)-M_\nu|<\frac 1 n$, we deduce that
\begin{equation}\label{third2}\begin{split}
-\frac{1}{2}\sum_{\nu=0}^{N} \Big(& \log r_{\nu,\tau(m_\nu)} -\log r_{\nu,\tau(m_{\nu-1})} \Big)= \frac{1}{2}\log \frac{a_0}{b_N} -\frac{1}{2}\sum_{\nu=0}^{N-1} \log \frac{b_\nu}{a_{\nu+1}} + \bigO(\frac{1}{n}).
\end{split}
\end{equation}
A combination of \eqref{app0} with \eqref{third1}, \eqref{third2} finishes the proof.
\end{proof}

\begin{lem} \label{ellcor}With $\ell(z)=2\log|z|$, we have
$\int_\C \ell\, d\sigma=-(q(b_N)-2\log b_N-q(a_0)).$
\end{lem}
\begin{proof}
Since $\int_\C\ell\, d\sigma= 2 \sum_{\nu=0}^N\int_{a_\nu}^{b_\nu}(\log s)2s\Delta Q(s)\, ds$, the statement follows from \eqref{expell1} and \eqref{expell2}.
\end{proof}

We continue by estimating the contribution coming from the perturbation $sh/n$, where $h$ is the smooth radially symmetric test-function in \eqref{begin}. (Here $s$ is an arbitrary real parameter; we will later choose it to satisfy $|s|\le \log n$.)

\begin{lem} \label{l37} As $n \to \infty$, we have, uniformly for $s\in\R$,
\begin{align*}
\sum_{\nu=0}^{N} \sum_{j=m_{\nu-1}}^{m_\nu-1}  sh(r_{\nu,\tau(j)})  &= n\int_{S} sh\, d\sigma + \frac{s(h(a_0)-h(b_N))}{2} + \sum_{\nu=0}^{N-1}  (x_\nu+\frac{1}{2})s(h(a_{\nu+1}) - h(b_{\nu}))  + \bigO(\frac{|s|}{n}).
\end{align*}
\end{lem}

\begin{proof}
The proof is similar to the proof of Lemma \ref{l22} and is omitted (one only has to substitute ``$\log\Delta Q(r)$'' in the proof of Lemma \ref{l22} by ``$sh(r)$'').
\end{proof}

We next study the contribution coming from the terms $a_{\nu,j}$ in \eqref{anjk}. To this end, we shall use an argument based on Riemann sums in \cite[Section 2]{BKS}, but with a new twist to account for the contribution due to the terms involving the perturbation $h$.

It is convenient to denote, for suitable functions $f$, (with $S^\nu=A(a_\nu,b_\nu)$)
\begin{align}\label{ehnu}
e_{\nu,f}&:=\frac{1}{8\pi} \int_{\partial S^\nu } \partial_n f\, |dz|-\frac{1}{8\pi}\int_{\partial S^\nu} f(z) \frac{\partial_n \Delta Q(z)}{\Delta Q(z)}\, |dz| +\frac{1}{2} \int_{S^\nu} f(z) \Delta \log \Delta Q(z)\, dA(z),\\
\label{vhnu}
v_{\nu,f}&:=\frac{1}{4}\int_{S^\nu} |\nabla f(z)|^2\, dA(z).
\end{align}

Also recall the definition of $F_Q[S^\nu]$ from \eqref{Fq_ann}.

\begin{lem} \label{putt} Fix $\nu$ with $0\le \nu\le N$. As $n\to\infty$ we have, uniformly for $|s|\leq \log n$,
\begin{equation}\label{rsum}\begin{split}
\frac{1}{n}\sum_{j=m_{\nu-1}}^{m_\nu-1}  a_{\nu,j} &=
F_Q[S^\nu] - \frac{1}{4} \log \frac{\Delta Q(b_{\nu})}{\Delta Q(a_\nu)}  +\frac{1}{3} \log  \frac{b_\nu}{a_\nu}   \\
	& +\frac{s}{2} (h(b_\nu)-h(a_\nu)) +se_{\nu,h}+\frac {s^2} 2 v_{\nu,h}+\bigO(\frac{{1+s^{2}}}{n}).\\
\end{split}
\end{equation}
\end{lem}
\begin{proof}
Recall the notation $d_\ell=d_{\ell,\nu,j}=g_{\tau(j)}^{(\ell)}(r_{\nu,\tau(j)})$. With $\calB(r)$ as in \eqref{br} we have
	\begin{equation}\label{sdep}\begin{split}
	a_{\nu,j} =& \calB(r_{\nu,\tau(j)})
	+ \frac{s^2h'(r_{k,\tau(j)})^2 }{2} \frac{1}{d_2}  + \frac{sh''(r_{k,\tau(j)})}{2}\frac{1}{d_2}+\frac{sh'(r_{k,\tau(j)})}{r_{k,\tau(j)}}\frac{1}{d_2}  -\frac{1}{2} sh'(r_{k,\tau(j)})\frac{d_3}{d_2^{2}}.
	\end{split}\end{equation}

We use a Riemann sum approximation followed by the observation that the inverse $\tau=\tau(r)$ to $r=r_{\nu,\tau}$ satisfies $\tau'(r)=2r\Delta Q(r)$ (see \eqref{difft}).
It follows that with $\sigma$ the equilibrium measure, we have
\begin{align*}\frac 1 n\sum_{j=m_{\nu-1}}^{m_\nu-1}\calB(r_{\nu,\tau(j)})&=\int_{M_{\nu-1}}^{M_\nu}\calB(r_{\nu,\tau})\, d\tau+\bigO(\frac 1 n)=\int_{S^\nu}\calB(r)\, d\sigma(z)+\bigO(\frac 1 n)\\
&=F_Q[S^\nu]- \frac{1}{4} \log\Big(\frac{\Delta Q(b_{\nu})}{\Delta Q(a_\nu)} \Big)+\frac{1}{3} \log\Big( \frac{b_\nu}{a_\nu}  \Big)+\bigO(\frac 1 n),\end{align*}
where the last equality follows from \cite[Lemma 2.2]{BKS}.

We next group together the terms in \eqref{sdep} which contain the parameter $s$.
Using a Riemann sum approximation, we find that these give the total contribution (for $|s|\le \log n$)
\begin{equation}\label{sim1}\begin{split}
&\frac 1 n \sum\limits_{j=m_{\nu-1}}^{m_{\nu}-1}\big[  \frac{s^2h'(r_{\nu,\tau(j)})^2 }{2} \frac{1}{d_2}  + \frac{sh''(r_{\nu,\tau(j)})}{2}\frac{1}{d_2}+\frac{sh'(r_{\nu,\tau(j)})}{r_{\nu,\tau(j)}}\frac{1}{d_2}  -\frac{1}{2} sh'(r_{\nu,\tau(j)})\frac{d_3}{d_2^{2}}\big] \\
 &=\int\limits_{M_{\nu-1}}^{M_{\nu}}  \frac{1}{2d_2} \Big( s^2 h'(r_{\nu,\tau})^2 + sh''(r_{\nu,\tau}) +2s\frac{h'(r_{\nu,\tau})}{r_{\nu,\tau}} - sh'(r_{\nu,\tau})\frac{d_3}{d_2}    \Big)\, d\tau+
  \bigO(\frac{|s|+s^2}{n})\\
  &=\int_{S^\nu}\frac 1 {4\Delta Q(r)}\Big(s^2 h'(r)^2+sh''(r)+2s\frac {h'(r)} r-sh'(r)\frac {\d_r\Delta Q(r)-r^{-1}\Delta Q(r)}
  {\Delta Q(r)}\big)\, d\sigma(z)\\
&  +
  \bigO(\frac{|s|+s^2}{n}),\\
\end{split}
\end{equation}
where we used
\eqref{d234}.

This further simplifies to
\begin{equation}\label{simpa}
\int\limits_{a_\nu}^{b_\nu}   \frac{r}{4} \Big( s^2 h'(r)^2 + sh''(r) +3s\frac{h'(r)}{r} - sh'(r) \frac{\partial_r \Delta Q(r)}{\Delta Q(r)}    \Big)\, dr + \bigO(\frac{|s|+s^2}{n}).
\end{equation}

In order to evaluate the integral in \eqref{simpa}, we use the following identities:
\begin{align}
& \frac{3s}{4} \int_{a_\nu}^{b_\nu} h'(r)\, dr = \frac{3s}{4} (h(b_\nu)-h(a_\nu)), \label{lol7}
\\
& \frac{s^2}{4} \int_{a_\nu}^{b_\nu} r h'(r)^2\, dr =  \frac{s^2}{2} \frac{1}{4}\int_{A(a_\nu,b_\nu)} |\nabla h(z)|^2\, dA(z)= \frac{s^{2}}{2}v_{\nu,h},
\label{lol8}
\\
& \frac{s}{4} \int_{a_\nu}^{b_\nu} r h''(r)\, dr = \frac{s}{4}\Big( b_\nu h'(b_\nu) - a_\nu h'(a_\nu)  \Big) -\frac{s}{4}\Big(h(b_\nu)-h(a_\nu) \Big), \label{lol9}
\\
& -\frac{s}{4} \int_{a_\nu}^{b_\nu} r h'(r) \frac{\partial_r \Delta Q(r)}{\Delta Q(r)}\, dr = -\frac{s}{8\pi} \int_{\partial S^\nu} h(z) \frac{\partial_n \Delta Q(z)}{\Delta Q(z)}\, |dz| + \frac{s}{2} \int_{S^\nu} h\, \Delta \log \Delta Q\, dA. \label{lol10}
\end{align}
Inserting these relations in \eqref{simpa}, we readily obtain \eqref{rsum}, finishing the proof.
\end{proof}

Until this point, we have naively summed according to blocks $\sum_{j=m_{\nu-1}}^{m_\nu-1} \log h_{\nu,j}$. For $j$ near the critical indices $m_{\nu-1}$, $m_\nu$ this approximation is not sufficiently accurate (compare \eqref{the text is hard to follow without this equation}).

The following lemma provides the necessary correction, and constitutes the point where the displacement term $\Osc_n$   enters the picture.

\begin{lem}\label{theta_contribution} For each $\nu$ with $0\le \nu\le N-1$, the term $T_\nu$ in \eqref{tndefn} obeys, as $n\to\infty$,
\begin{align*}
T_\nu
&= \log [(-\mu_\nu\rho_\nu;\rho_\nu^2)_\infty ] +\log [(-\mu_\nu^{-1}\rho_\nu;\rho_\nu^2)_\infty]
+\bigO(\frac {1+|s|
} n).
\end{align*}
Here $\rho_\nu$ and $\mu_\nu=\mu_\nu(s)$ are given by \eqref{rdef2} and the $\bigO$-constant is uniform  for $|s| \leq \log n$.
\end{lem}

\begin{proof}
	From Lemma \ref{qnormlem} we have, as $\tau(j)\to M_\nu$, with $m_\nu=\lfloor M_\nu n\rfloor$,
$$
	\frac{h_{\nu+1,j}}{h_{\nu,j}} = \frac{1}{\rho_\nu} \mu_\nu(s) \rho_\nu^{2(m_\nu -j)})\cdot \Big[1+ \bigO((1+|s|)(\tau(j)-M_\nu)) + \bigO(n\cdot (\tau(j)-M_\nu)^2)\Big].
$$
Let us write
$$
\tilde{T}_\nu := \sum\limits_{j=m_\nu}^{m_\nu+ L_n}  \log(1 + \frac{1}{\mu_\nu}\rho_\nu^{2(j-m_\nu)+1}   )   + \sum\limits_{j=m_\nu-L_n}^{m_{\nu}-1} \log(1+ \mu_\nu \rho_\nu^{2(m_\nu-j )-1 }).
$$
Using that $m_0(n)/n\to M_0>0$ while $\rho_\nu<1$ and
$n^{-c}\le \mu_\nu\le n^c$ for a suitable $c\ge 0$ (see \eqref{mudef}), we deduce the following estimate for \eqref{tndefn}, as $n\to\infty$
	\begin{equation}\label{tt}
T_\nu= \tilde{T}_\nu\cdot \bigg(1 + \bigO\bigg( \frac{1+|s|}{n}\sum_{j=0}^{L_{n}}j\rho^{j} \bigg)\bigg) = \tilde{T}_\nu\cdot \bigg(1 + \bigO\bigg( \frac{1+|s|}{n} \bigg)\bigg).
	\end{equation}
	
By changing the summation index, we have
	$$
	\tilde{T}_\nu=\sum\limits_{j=0}^{L_n}  \log(1 + \frac{1}{\mu_\nu}\rho_\nu^{2j+1}   )   + \sum\limits_{j=0}^{L_n-1} \log(1+ \mu_\nu \rho_\nu^{2(j+1 )-1 }) .
	$$
Using \eqref{limpoch}, we recognize that
	$$
	\tilde{T}_\nu = \log [(-\mu_\nu\rho_\nu;\rho_\nu^2)_\infty ] +\log [(-\mu_\nu^{-1}\rho_\nu;\rho_\nu^2)_\infty ] +\bigO(\rho_\nu^{2L_n}).
	$$

The desired asymptotic now follows from \eqref{tt} and the fact that $\tilde{T}_\nu=\bigO(1)$.
\end{proof}

 \begin{proof}[Proof of Theorem \ref{regular_expansion} for annular droplets] To derive the asymptotic for $\log Z_{n,sh}=\log(n!)+\sum_{0}^{n-1}\log h_j$ we first sum by blocks
 $\sum_{\nu=0}^N\sum_{j=m_{\nu-1}}^{m_\nu-1}\log h_{\nu,j}$ where
\begin{equation}\label{spit}\begin{split}\log h_{\nu,j}= &\frac 1 2 \log \frac {2\pi} n-ng_{\tau(j)}(r_{\nu,\tau(j)})-\frac 1 2 \log\Delta Q(r_{\nu,\tau(j)})+\log r_{\nu,\tau(j)}\\
&+sh(r_{\nu,\tau(j)})+\frac { a_{\nu,j}}n+\bigO(\frac {1} {n^2}).\\
\end{split}
\end{equation}

Summing the terms in the right side of \eqref{spit} using Lemmas \ref{l1}--
\ref{putt}, correcting for the $j$ near the critical indices $m_\nu$ by means of \eqref{the text is hard to follow without this equation} and Lemma \ref{theta_contribution}, and adding also $\log(n!)$, using Stirling's approximation
$\log(n!)=n\log n-n+\frac 1 2\log n+\frac 1 2 \log (2\pi)+\bigO(n^{-1})$, it is straightforward to finish the proof.
\end{proof}

\section{Central disk droplet with a Fisher-Hartwig singularity}\label{reg2}

We shall now adapt our proof in the previous section to the situation with a central disk component $\D_{b_0}=A(0,b_0)$ (sometimes denoted $S_0$). We therefore adopt the notation from Section \ref{Sec_Reg}, except that we now
assume $a_0=0$, and we allow the parameter $\alpha$ to be non-zero.

Our goal is to prove Theorem \ref{conical_expansion}, and, as a consequence, to prove the central disk part of
 Theorem \ref{regular_expansion}.

\subsection{Approximation scheme for the disk case}

Recall that $h_j$ is the squared norm
\begin{equation}\label{remind}
h_j=\|p_j\|^2=2\int_0^\infty r^{1+2\alpha}e^{sh(r)}e^{-ng_{\tau(j)}(r)}\, dr,
\end{equation}
where $\tau(j)=\frac j n$, $g_\tau(r)=q(r)-2\tau\log r$, and $q(r)=Q(r)$.

One verifies, without any changes in the proof, that the approximation $h_j=h_j^\sharp\cdot (1+\bigO(n^{-100}))$ in \eqref{norm} and \eqref{hjsharp} remains valid
for $j\ge m_0$, i.e., for components $S^\nu=A(a_\nu,b_\nu)$ with $\nu\ge 1$.

It therefore remains to study the sum $\sum_{j=0}^{m_0-1}\log h_j$. Following an idea in \cite{BKS}, we shall split the sum at $j=D_n$ where we take
\begin{equation}\label{dndef}D_n:=\lceil n^{\frac{1}{6}}\rceil.\end{equation}

We shall use different estimates for the two sums
\begin{align*}
\sum_{j=0}^{D_n-1}\log h_j,\qquad \text{and}\qquad \sum_{j=D_n}^{m_0 -1}\log h_j.
\end{align*}

We start with the following lemma.

\begin{lem}\label{low_sum1}
As $n\to \infty$, for $j=0,1,...,D_n-1$ and $|s|\le \log n$ we have
\begin{align*}
\log h_j = -n q(0) -(j+1+\alpha) \log  (n\Delta Q(0))  &+ \log \Gamma(j+1+\alpha)  + sh(0) + \bigO(n^{-\frac{1}{2}} (j+1)^{\frac{3}{2}} (\log n)^{3}).
\end{align*}
\end{lem}

\begin{proof}
In the case $s=0$ and $\alpha=0$, this is just \cite[Lemma 3.1]{BKS}. The adaptations for general $s$ and $\alpha$ are straightforward and we omit details.
\end{proof}

The lemma gives the estimate
\begin{align*}
\sum_{j=0}^{D_n-1}\log h_j =& -nD_n q(0)-\frac {D_n(D_n+1+2\alpha)}2 \log(n\Delta Q(0))+D_n sh(0) + \log \frac{G(D_n+1+\alpha)}{G(1+\alpha)}+\calE_n,
\end{align*}
where $G$ is Barnes $G$-function and $\calE_n=\bigO(\frac {D_n^{\frac 5 2}(\log n)^{3}} {\sqrt{n}})$. Using well-known asymptotics
for $G(n+1)$ in \cite[Eq. (5.17.5)]{NIST} we deduce the following result; a simple generalization of the special case $s=0$ and $\alpha=0$ given in \cite[Lemma 3.3]{BKS}.

\begin{lem} \label{low_sum} For $|s|\le \log n$ we have, as $n\to\infty$,
\begin{align*}
\sum\limits_{j=0}^{D_n-1} \log  h_j =& -D_n n  q(0)-\frac{D_n(D_n+1+2\alpha)}{2}\log (n\Delta Q(0))\\
&+ \frac{1}{2}D_n^2 \log D_n - \frac{3}{4} D_n^2 + \alpha D_{n} \log D_{n} + D_n \bigg(  s h(0) + \frac{\log(2\pi)}{2} - \alpha  \bigg)\\
  &+  \frac{6\alpha^{2}-1}{12} \log D_n  +  \frac{\alpha}{2} \log(2\pi)   + \zeta'(-1)  -  \log G(1  +  \alpha)+\calE_n,
\end{align*}
where $\zeta(s)$ is Riemann's zeta function and $\calE_n=\bigO(n^{-\frac 1 {12}}(\log n)^3)$.
\end{lem}

Next we recall the notation
$$\ell(z)=2\log|z|,\qquad k(z)=sh(z)+\alpha\ell(z).$$

Assume that $j$ satisfies $D_n\le j\le m_0+L_n$, where $L_n=C\log n$ and let
$$r=r_{0,\tau(j)}$$ be the solution to $rq'(r)=2\tau(j)$ with $r$ in the vicinity of the interval $[0,b_0]$, where $\tau(j)=j/n$.

We write $h_{0,j}$ for the approximation to $h_j$
in \eqref{norm} but with the perturbation ``$sh$'' replaced by ``$k$'', i.e.,
\begin{equation}h_{0,j}:=2\int_{J_{0,j}}r^{1+2\alpha}e^{sh(r)}e^{-ng_{\tau(j)}(r)}\, dr,\end{equation}
where $J_{0,j}=\{r\ge 0\,;\, |r-r_{0,\tau(j)}|<\eps_n\}$ and $\eps_n$ as in \eqref{dnen}.

The Laplace method in Lemma \ref{explicit1} gives the following approximation.

\begin{lem} \label{bkshigh} For $D_n\le j\le m_0+L_n$, we have
\begin{equation}\label{normzero}
h_{0,j}=\sqrt{\frac {2\pi} n} \frac {r_{0,\tau(j)}} {\sqrt{\Delta Q(r_{0,\tau(j)})}}  e^{k(r_{0,\tau(j)})}e^{-ng_{\tau(j)}(r_{0,\tau(j)})}\cdot
\left(1+\frac {a_{0,j}} n+\epsilon\right),\end{equation}
where the term $a_{0,j}$ is given in \eqref{anjk} but with $sh$ replaced by $k=sh+\alpha\ell$, i.e.,
\begin{align}\label{correct}a_{0,j}=\calB(r_{0,\tau(j)})
+ \frac{k'(r_{0,\tau(j)})^2 }{2} \frac{1}{d_{2}}  + \frac{k''(r_{0,\tau(j)})}{2}\frac{1}{d_{2}} +\frac{k'(r_{0,\tau(j)})}{r_{\nu,\tau(j)}}\frac{1}{d_{2}}  -\frac{k'(r_{0,\tau(j)})}{2} \frac{d_{3}}{d_{2}^{2}},
\end{align}
where $\mathcal{B},d_2,d_3$ are defined as in \eqref{br},\eqref{dlnk}, while
$$\epsilon=\bigO(j^{-\frac{3}{2}} (\log n)^{c} ),$$
where $c>0$ may be chosen arbitrarily small.

For ``large'' $j$, say $j\ge c_0n$ where $c_0>0$, \eqref{normzero} holds with the better bound $\epsilon=\bigO(n^{-2})$.
\end{lem}

\begin{proof}
The case $s=\alpha=0$ is found in \cite[Lemma 3.2]{BKS}. The adaptations needed for the general case are straightforward, and are omitted.
\end{proof}

Moreover, if for some $\nu\ge 1$ we have $m_{\nu-1}-L_n\le j\le m_\nu+L_n$, we write $r_{\nu,\tau(j)}$ for the unique solution to $rq'(r)=2\tau(j)$ in the vicinity of the interval $a_\nu\le r\le b_\nu$. In this situation we write
\begin{equation}h_{\nu,j}:=2\int_{J_{\nu,j}}r^{1+2\alpha}e^{sh(r)}e^{-ng_{\tau(j)}(r)}\, dr,\end{equation}
where $J_{\nu,j}=\{r\ge 0\,;\, |r-r_{\nu,\tau(j)}|<\eps_n\}$. In this case, the analogue of \eqref{normzero} holds with $\epsilon=\bigO(n^{-2})$, and the formula \eqref{correct} holds with the index ``$0$'' replaced by ``$\nu$''.

For $D_n\le j\le n-1$ we next define $h_j^\sharp$ by $h_j^\sharp=h_{\nu,j}$ if $m_{\nu-1}+L_n<j<m_\nu-L_n$ and $h_j^\sharp=h_{\nu,j}+h_{\nu+1,j}$ if $|j-m_\nu|\le L_n$. (Here $h_{N+1,j}:=0$.)

The following result is a direct consequence of Lemma \ref{loc_lem}, Lemma \ref{bkshigh} and the above argument.

\begin{lem} For $D_n\le j\le n-1$, the squared norm \eqref{remind} satisfies $h_j=h_j^\sharp\cdot (1+\bigO(n^{-100}))$.
\end{lem}

The next lemma is somewhat similar to \cite[Lemma 3.7]{BKS}, but differs sufficiently to require a new analysis.

\begin{lem}\label{high_sum} With $k=sh+\alpha\ell$ and notation as in Section \ref{fssub}, we have, as $n\to\infty$,
\begin{align*}\sum\limits_{j=D_n}^{n-1} &\log h_{j}=-n^2I_{Q}[\sigma]+\frac {n-D_n}2\log\frac {2\pi} n-\frac n 2E_Q[\sigma]+F_Q[\sigma]+n\int_\C k\, d\sigma\\
&+e_k+\frac {s^2} 2 v_{0,h}+\frac 12 \sum_{\nu=1}^Nv_{\nu,k}+\Osc_n(s,\alpha)+\alpha s (h(b_{0})-h(0))\\
&+nD_nq(0)+\frac 3 4 D_n^2-\frac 1 2(D_n^2+D_n)\log \frac 1 {\Delta Q(0)}-\frac 1 2\left(D_n^2-\frac 1 6\right)\log\frac {D_n} n-D_nsh(0)\\
&-\alpha D_n\log\left(\frac {D_n}{n\Delta Q(0)}\right)+\alpha D_n-\frac {\alpha^2} 2\log\frac {D_n}n+\frac {\alpha^2}2\log (b_0^2\Delta Q(0)) +\calE_n,
\end{align*}
where $\calE_n=\bigO(n^{-\frac 1 {12}}(\log n)^3)$.
\end{lem}

Combining Lemma \ref{low_sum} and \ref{high_sum} with estimates from the previous section, it is fairly straightforward to deduce asymptotics for $\log Z_n$.
However, we begin by giving a detailed proof of
Lemma \ref{high_sum} in the next subsection.

\subsection{Proof of the key lemma} We turn to our proof of Lemma \ref{high_sum}. The proof involves new estimations of terms in $\sum_{D_n}^{m_0-1}\log h_{0,j}$, where $D_n$ is given in \eqref{dndef}. Such terms were already treated in our analysis of the annular case, but new subtleties arise since the radii $r_{0,j}$ are not bounded from below in the present case, and also because we allow a Fisher-Hartwig singularity, i.e., $\alpha\ne 0$.

We start with the following estimate. Using earlier notation (from \eqref{gnt}) we write
$$\gamma_0(t)=g_{\tau(t)}(r_{0,\tau(t)}),$$
where $t$ is in $0\le t\le M_0n+\bigO(\log n)$.
As before: $\tau(t)=t/n$, $g_\tau(r)=q(r)-2\tau\log r$ and $r=r_{0,\tau}$ is the solution to $rq'(r)=2\tau$ with $0\le r< b_0 +o(1)$.

\begin{lem}\label{main_term} As $n\to\infty$,
	\begin{align*}
	-n\sum\limits_{j=D_n}^{m_0-1} &g_{\tau(j)}(r_{0,\tau(j)}) = -n^{2}\int\limits_{\mathbb{D}_{b_0}} Q\, \Delta Q\, dA - \frac{n^{2}}{4}\int\limits_{0}^{b_0} sq'(s)^2\, ds + n^{2}M_0^2 \log b_0\\
 &+n\int_{m_0}^{M_0n}\gamma_0(t)\, dt+\frac {n} 2 \gamma_0(m_0)-\frac {n} {12}\gamma_0'(m_0)-n\frac {q(0)} 2+nD_n q(0)\\
   &
   +\frac 3 4 D_n^2
 - \frac 1 {2}(D_n^2-D_n+\frac 1 6)\log\Big(\frac{D_n}{n\Delta Q(0)} \Big)
	  - \frac{D_n}{2}  + \bigO(D_n^{-2}+n^{-1/2} D_n^{5/2}).
 	\end{align*}

\end{lem}
\begin{proof}

Using the Euler-Maclaurin formula (Theorem \ref{em_formula} with $d=2$) in a similar way as in the proof of Lemma \ref{l1}, we have
	\begin{equation}\label{em3}\begin{split}
	\sum\limits_{j=D_n}^{m_0-1} g_{\tau(j)}(r_{0,\tau(j)}) =& \int\limits_{D_n}^{m_0} \gamma_0(t)\, dt-\frac{1}{2}\Big(\gamma_{0}(m_{0})-\gamma_{0}(D_n)\Big) + \frac{1}{12}\Big( \gamma'_{0}(m_{0})-\gamma'_{0}(D_n) \Big)
 + \epsilon,\\
	\end{split}
\end{equation}
where $\epsilon$ shall be estimated.

We begin by studying the integral in \eqref{em3}. Using the substitution $u=r_{0,\tau(t)}$ from \eqref{change_var}, we deduce
 that
 	\begin{align}\label{cov3}
	\int\limits_{D_n}^{m_0} \gamma_0(t)\, dt &= 2n (\int\limits_{0}^{b_0}-\int\limits_{0}^{r_n}) (q(s)-sq'(s)\log s) s\Delta Q(s) ds - \int\limits_{m_0}^{M_0 n} \gamma_0(t)\, dt
	\end{align}
where we put $r_n:=r_{0,\tau(D_n)}.$ As in the proof of Lemma \ref{l1} we have	
\begin{equation}\label{cov4}\begin{split}
2n\int_{0}^{b_0} (q(s)-sq'(s)\log s) s\Delta Q(s)\, ds = n\int_{\mathbb{D}_{b_0}} &Q\, \Delta Q\, dA + \frac{n}{4}\int_{0}^{b_0} sq'(s)^2\, ds - n M_0^2 \log b_0.\end{split}
\end{equation}

Next note that our assumption $a_0=0$ implies that $Q(z)$ has a local minimum at $z=0$, so $q'(0)=0$ and $4\Delta Q(0)=4\lim_{r\to 0}(q''(r)+\frac {q'(r)-q'(0)}{r-0})= 2q''(0)$, which is strictly positive.

Recalling that $$\frac {d r_{0,\tau}} {d\tau}=\frac 1 {2r_{0,\tau}\Delta Q(r_{0,\tau})}$$ (see \eqref{difft}),  we find
$\Delta Q(0)\frac {d (r_{0,\tau}^2)}{d\tau}=1+\bigO(\tau)$ as $\tau \to 0_+$, and so, as $r_{0,0}
=0$,
\begin{equation}\label{smalltau}
r_{0,\tau}=\sqrt{\frac \tau {\Delta Q(0)}}+\bigO(\tau),\qquad (\tau\to 0_{+}).
\end{equation}

We can now estimate the error term $\epsilon$ in \eqref{em3}: using $\gamma_0'(t)=-\frac 2 n \log r_{0,\tau(t)}$
we deduce that $\gamma_0''(D_n)=\bigO(\frac 1 {n^2\tau(D_n)})$, $\gamma_0^{(3)}(t)=\bigO(
\frac 1 {n^3\tau(t)^{2}})$, and $$|\gamma_0^{(4)}(t)|\lesssim\frac 1 {n^4\tau(t)^{3}}=\frac 1 {nt^3},\qquad \text{for}\quad t\ge D_n.$$ Hence by Theorem \ref{em_formula} we have $|\epsilon|\lesssim \frac{1}{n^{3}\tau(D_{n})^{2}} = \frac 1 {nD_n^2}$.

	As a consequence, we deduce the approximations as $\tau\to 0_+$,
	\begin{equation}\label{app2}\begin{split}
	q(r_{0,\tau}) &=  q(0) + \tau  + \bigO(\tau^{3/2}),\quad \text{and} \quad \log r_{0,\tau} = \frac{1}{2} \log \Big(\frac{\tau}{\Delta Q(0)} \Big) + \bigO(\tau^{1/2}).\\
	\end{split}
\end{equation}

	 Integrating as in \eqref{cov4} over the interval from $0$ to $r_n$, and then integrating by parts, we deduce that
\begin{equation}\label{ibp}\begin{split}2n&\int_0^{r_n}(q(s)-sq'(s)\log s) s\Delta Q(s)\, ds =\frac n 2\int_0^{r_n}sq(s)(q''(s)+\frac {q'(s)} s)\, ds+\frac{n}{4}\int\limits_{0}^{r_n} sq'(s)^2\, ds\\
& - n[\tau(D_n)]^2 \log r_n=\frac n 2 r_nq(r_n)q'(r_n)
-\frac n 4\int_0^{r_n}sq'(s)^2\, ds- n[\tau(D_n)]^2 \log r_n.
\end{split}\end{equation}

Using that $nr_nq'(r_n)=2n\tau(D_n)=2D_n$ and for small $s$,  $q'(s)=q''(0)s+\bigO(s^2)= 2{\Delta Q(0)} s+\bigO(s^2)$, we first have
\begin{align}\label{ins0}
\frac{n}{4} \int\limits_{0}^{r_n} sq'(s)^2\, ds &= \frac 1 4 n\Delta Q(0)^2 r_n^4 + \bigO(nr_n^5)=\frac 1 4 \frac {D_n^2}
{n}+\bigO(nr_n^5),
\end{align}
and then (using \eqref{app2} and again $r_n q'(r_n)=2\tau(D_n)$)
that
\begin{equation}\label{com3}\begin{split}
2n\int\limits_{0}^{r_n} (q(s)-sq'(s)\log s) &s\Delta Q(s)\, ds  = D_nq(0) + \frac 3 4 \frac{D_n^2}{n}\\
  &- \frac 1 2 \frac{D_n^2}{n}\log\left[\frac{D_n}{n\Delta Q(0)}\right] +\bigO(n[\tau(D_n)]^{5/2}).\\
\end{split}\end{equation}

We also note that
\begin{equation}\label{com4}\begin{split}
\gamma_0(D_n) &= q(r_n)-2\tau(D_n)\log r_n \\
&= q(0) + \tau(D_n)  -\tau(D_n)\log\Big( \frac{\tau(D_n)}{\Delta Q(0)}\Big) + \bigO([\tau(D_n)]^{3/2}) ,\\
\end{split}
\end{equation}
and
\begin{align}\label{com5}
\gamma_0'(D_n) = -\frac{2}{n}\log r_n = -\frac{1}{n}\log\Big( \frac{\tau(D_n)}{\Delta Q(0)}\Big) + \bigO(n^{-1}[\tau(D_n)]^{1/2}) .
\end{align}

A combination of \eqref{em3}-\eqref{cov4} and \eqref{com3}-\eqref{com5} finishes the proof.
\end{proof}

Before proceeding, it is convenient to note the following lemma, which is the natural counterpart of Lemma \ref{ellcor} for central disk droplets. The proof
generalizes immediately, by setting $a_0=0$.

\begin{lem}\label{elllem} With $\ell(z)=2\log|z|$ we have the formula
$$\int\ell\, d\sigma=-(q(b_N)-2\log b_N-q(0)).$$
\end{lem}

The following lemma gives a counterpart to \cite[Lemma 3.4]{BKS} (cf.~also Lemma \ref{l1} above).

\begin{lem} \label{mainterm+}
With $x_\nu=M_\nu n-m_\nu$ we have as $n\to\infty$
\begin{align*}-n\sum_{j=D_n}^{m_0-1}g_{\tau(j)}&(r_{0,\tau(j)})-n\sum_{\nu=1}^{N}\sum_{j=m_{\nu-1}}^{m_\nu-1}g_{\tau(j)}(r_{\nu,\tau(j)})\\
&=-n^2I[\sigma]-\frac n 2 \int\ell\, d\sigma
 +\frac 1 6\log b_N  +\sum_{\nu=0}^{N-1}\left(x_\nu^2+x_\nu+\frac 1 6\right)\log\rho_\nu+nD_n q(0)
   +\frac 3 4 D_n^2\\
&\quad
 - \frac 1 {2}(D_n^2-D_n+\frac 1 6)\log\Big(\frac{D_n}{n\Delta Q(0)} \Big)
		  - \frac{D_n}{2}+\bigO(D_n^{-2}+n^{-\frac 1 2}D_n^{\frac 5 2}).
\end{align*}
\end{lem}

\begin{proof} In the first sum $\sum_{j=D_n}^{m_0-1} g_{\tau(j)}(r_{0,\tau(j)})$ we use the approximation in Lemma \ref{main_term} and in the second one $\sum_{\nu=1}^{N}\sum_{j=m_{\nu-1}}^{m_\nu-1}g_{\tau(j)}(r_{\nu,\tau(j)})$ we use the approximation from (the proof of) Lemma \ref{l1}: for each $\nu$ with $1\le\nu\le N$ we have
\begin{align*}\sum_{j=m_{\nu-1}}^{m_\nu-1}g_{\tau(j)}(r_{\nu,\tau(j)})&=nI_{Q,\nu}[\sigma]+\int_{m_{\nu-1}}^{M_{\nu-1}n}\gamma_\nu(t)\,dt-\int_{m_\nu}^{M_\nu n}\gamma_\nu(t)\, dt\\
&-\frac 12(\gamma_\nu(m_\nu)-\gamma_\nu(m_{\nu-1}))+\frac 1 {6n}\log \rho_\nu+\bigO(n^{-2}).
\end{align*}

By straightforward manipulations using Lemma \ref{W-en-lem}, the estimates \eqref{rest}, \eqref{2est} and the following easily verified counterpart to \eqref{3est},
\begin{align*}\frac 1 {12}\gamma_0'(m_0)+&\frac 1 {12}\sum_{\nu=1}^N(\gamma_\nu'(m_\nu)-\gamma_\nu'(m_{\nu-1}))=-\frac 1 {6n}\log b_N-\frac 1 {6n}\sum_{\nu=0}\log \rho_\nu+\bigO(n^{-2}),
\end{align*}
one now finishes the proof of the lemma. We omit details.

\end{proof}

In the next two lemmas, we adapt the computations in \cite[Lemma 3.5]{BKS} to the situation with gaps.

\begin{lem}\label{lap_term} As $n\to\infty$,
\begin{align*}\sum_{j=D_n}^{m_0-1}&\log \Delta Q(r_{0,\tau(j)})+\sum_{\nu=1}^{N}\sum_{j=m_{\nu-1}}^{m_\nu-1}\log \Delta Q(r_{\nu,\tau(j)}) = nE_Q[\sigma]+\frac 12\log\frac {\Delta Q(0)}{\Delta Q(b_N)}\\
&-\sum_{\nu=0}^{N-1}(x_\nu+\frac 12)\log\frac {\Delta Q(b_\nu)}{\Delta Q(a_{\nu+1})} -D_n\log\Delta Q(0)+\bigO(n^{-\frac 1 2}D_n^{\frac 3 2}).
\end{align*}
\end{lem}

\begin{proof} We adapt the proof of Lemma \ref{l22}. To handle the first sum on the left, we use the Euler-Maclaurin formula in the form
\begin{equation}\label{em4}\begin{split}
	\sum\limits_{j=D_n}^{m_0-1} \log &\Delta Q (r_{0,\tau(j)}) = \int\limits_{D_n}^{m_0} \log \Delta Q(r_{0,\tau(t)})\, dt\\ & -\frac{1}{2}\Big(\log \Delta Q(r_{0,\tau(m_0)})
-\log\Delta Q(r_{0,\tau(D_n)}) \Big) +\bigO((nD_n)^{-\frac 1 2}).\\
\end{split}
\end{equation}

The error term in \eqref{em4} follows using appropriate bounds on some derivatives of the function $p_0(t):=\log \Delta Q(r_{0,\tau(t)})$. Indeed, using $\frac {d r_{0,\tau}}{d\tau}=\frac 1 {2r_{0,\tau}\Delta Q(r_{0,\tau})}$ and
$r_{0,\tau}=\bigO(\tau^{1/2})$ as $\tau\to 0_{+}$, we obtain
$$p_0'(D_n)\asymp \frac 1 {n\tau(D_n)^{1/2}}\qquad \text{and}\qquad \int_{D_n}^{m_0}|p_0''(t)|\, dt\lesssim \frac {\log n} n.$$
This gives the desired error term, in view of Theorem \ref{em_formula} with $d=1$.

The remaining terms in the right side of \eqref{em4} are estimated as follows. First note that
$$
\int\limits_{0}^{m_0} \log \Delta Q(r_{0,\tau(t)})\, dt = n \int\limits_{\D_{b_0}} \log \Delta Q\, d\sigma -x_0 \log \Delta Q(b_0) + \bigO(\frac{1}{n}),
$$
and, setting $r_n=r_{0,\tau(D_n)}$,
\begin{align*}\int\limits_{0}^{D_n} \log \Delta Q(r_{0,\tau(t)})\, dt &= 2n\int\limits_{0}^{r_n}  ( \log \Delta Q(s) ) \Delta Q(s) s\, ds =n\int_{\D_{r_n}}(\log\Delta Q)\, d\sigma\\
&=n\tau(D_n)\log \Delta Q(0)+\bigO(n[\tau(D_n)]^{\frac 3 2}).
\end{align*}
Adding up,
\begin{equation}\label{addu}
\int_{D_n}^{m_0}\log \Delta Q(r_{0,\tau(t)})\, dt=n\int\limits_{\mathbb{D}_{b_0}} \log \Delta Q\, d\sigma-x_0\log\Delta Q(b_0)-D_n\log\Delta Q(0) +\bigO(D_n^{\frac 32}n^{-\frac 12}).
\end{equation}
We also note that
\begin{equation}\label{addu2}\log \Delta Q(r_{0,\tau(m_0)})
-\log\Delta Q(r_n)=\log\frac {\Delta Q(b_0)}{\Delta Q(0)}+\bigO(r_n).
\end{equation}
Combining \eqref{em4}, \eqref{addu}, \eqref{addu2}, we have
\begin{equation}\label{sill}
\begin{split}
\sum\limits_{j=D_n}^{m_0-1} \log \Delta Q (r_{0,\tau(j)}) =& n \int\limits_{\mathbb{D}_{b_0}} \log \Delta Q\, d\sigma -\frac{1}{2}\log\frac{\Delta Q(b_0)}{\Delta Q(0)}
\\ & - x_0 \log \Delta Q(b_0) -D_n \log \Delta Q(0) + \bigO(D_n^{\frac 12}n^{-\frac 12}).
\end{split}
\end{equation}
We now add the estimates from the proof of Lemma \ref{l22}: for $1\le \nu\le N$,
\begin{align*}\sum_{j=m_{\nu-1}}^{m_\nu-1}&\log\Delta Q(r_{\nu,\tau(j)})=n\int_{S^\nu}\log\Delta Q\, d\sigma+x_{\nu-1}\log\Delta Q(a_\nu)\\
&-x_\nu\log\Delta Q(b_\nu)-\frac 1 2 (\log \Delta Q(a_\nu)-\log \Delta Q(b_{\nu-1}))+\bigO(n^{-1}),
\end{align*}
where we recall that $S^{\nu}=A(a_{\nu},b_{\nu})$. Adding these equations, we finish the proof of the lemma.
\end{proof}

\begin{lem} \label{rad_term} As $n\to\infty$,
\begin{equation}\label{llc}\begin{split}\sum_{j=D_n}^{m_0-1}&\log r_{0,\tau(j)}+\sum_{\nu=1}^N \sum_{j=m_{\nu-1}}^{m_\nu-1}\log r_{\nu,\tau(j)}=\frac n 2\int\ell\, d\sigma-\frac 1 2\log b_N\\
&-\sum_{\nu =0}^{N-1}(x_\nu+\frac 1 2)\log\rho_\nu-(\frac {D_n}2-\frac 14)\log \Big(\frac {D_n}{n\Delta Q(0)}\Big)+ \frac{1}{2}D_n +\bigO(D_n^{-1}+n^{-\frac 12}D_n^{\frac 32}).
\end{split}
\end{equation}
\end{lem}

\begin{proof} 	Using the Euler-Maclaurin formula as in the proof of Lemma \ref{l36}
	\begin{equation}\label{em5}
	\sum\limits_{j=D_n}^{m_0-1} \log r_{0,\tau(j)} = \int\limits_{D_n}^{m_0} \log r_{0,\tau(t)}\, dt -\frac{1}{2} \Big(\log r_{0,\tau(m_0)}-\log r_{0,\tau(D_n)}  \Big) + \bigO(D_n^{-1} ),
	\end{equation}
where the error term is obtained by a similar computation as in Lemma \ref{lap_term}.

Continuing as in the proof of Lemma \ref{l36}, we integrate by parts and obtain (using Lemma \ref{elllem})
	\begin{equation}\label{ibp5}\begin{split}
	\int\limits_{0}^{M_0 n} \log r_{0,\tau(t)}\, dt &= 2n \int\limits_{0}^{b_0} s\Delta Q(s) \log s\, ds =\frac n 2 \int\ell\, d\sigma,
\end{split}
	\end{equation}
	and
	\begin{equation}\label{rem5}
	-\int\limits_{m_0}^{M_0 n} \log r_{0,\tau(t)}\, dt = -x_0 \log b_0 + \bigO(\frac{1}{n}).
	\end{equation}

Recalling that $q'(0)=0$, $q''(0)=2\Delta Q(0)$, $r_{0,\tau} = \sqrt{ \frac{\tau}{\Delta Q(0)}} + \bigO(\tau)$,
and setting
$r_n=r_{0,\tau(D_n)}$,
	we also deduce that
\begin{equation}\begin{split}\label{rem6}
\int\limits_{0}^{D_n} &\log r_{0,\tau(t)}\, dt  = 2n\int_0^{r_n}s\Delta Q(s)\log s\, ds=\frac n 2 r_nq'(r_n)\log r_n-\frac n 2 \int_0^{r_n}q'(s)\, ds  \\
&=D_n\log r_n-\frac n 2 q''(0)(\frac{r_n^2}2+\bigO(r_n^3))=\frac 1 2 D_n\log \frac {\tau(D_n)}{\Delta Q(0)}-\frac {D_n} 2+\bigO(D_n^{3/2}/\sqrt{n}).
\end{split}
\end{equation}
Finally,
$$-\frac 1 2 (\log r_{0,\tau(m_0)}-\log r_{0,\tau(D_n)} )=-\frac 1 2 \log b_0+\frac 1 4 \log \frac {\tau(D_n)}{\Delta Q(0)}+\bigO(\sqrt{\tau(D_n)}).$$

Inserting \eqref{ibp5}, \eqref{rem5}, \eqref{rem6} in \eqref{em5}, we find the asymptotic formula
	\begin{equation}\label{radd}\begin{split}
	\sum\limits_{j=D_n}^{m_0-1} \log r_{0,\tau(j)} =& \frac n 2 (2M_0\log b_0+q(0)-q(b_0)) -(\frac {D_n}2-\frac 14)\log \Big(\frac {D_n}{n\Delta Q(0)}\Big)\\
&-(x_0+\frac{1}{2})\log b_0+ \frac{1}{2}D_n +\bigO(D_n^{-1}) .
	\end{split}\end{equation}

For $\nu\ge 1$, we have by Lemma \ref{l36} (and its proof)
\begin{equation}\label{sadd}\begin{split}\sum_{j=m_{\nu-1}}^{m_\nu-1}\log r_{\nu,\tau(j)}&=-\frac 1 2 g_{M_\nu}(r_{\nu,M_\nu})+\frac 12g_{M_{\nu-1}}(r_{\nu,M_{\nu-1}})\\
&-x_\nu\log b_\nu+x_{\nu-1}\log a_\nu
-\frac 1 2 (\log b_\nu-\log a_\nu)+\bigO(n^{-1}).\\
\end{split}\end{equation}

Adding \eqref{radd} and \eqref{sadd} for $\nu=1,\ldots,N$ and using
$g_{M_{\nu-1}}(r_{\nu-1,M_{\nu-1}})=g_{M_{\nu-1}}(r_{\nu,M_{\nu-1}}),$
 we conclude the proof of the lemma.
\end{proof}

We next consider the contribution coming from the perturbation $sh/n$.

\begin{lem}\label{pert_term} For $|s|\le \log n$ we have, as $n\to\infty$,
\begin{align*}\sum\limits_{j=D_n}^{m_0-1} sh(r_{0,\tau(j)})&+\sum_{\nu=1}^N \sum_{j=m_{\nu-1}}^{m_\nu-1}sh(r_{\nu,\tau(j)})=ns\int_\C h\, d\sigma+\frac 1 2 s(h(0)-h(b_N))\\
&+\sum_{\nu=0}^{N-1}(x_\nu+\frac 12)s(h(a_{\nu+1})-h(b_\nu))-D_n sh(0)+\bigO(n^{-\frac 1 2}D_n^{\frac 3 2}\cdot |s|).
\end{align*}
\end{lem}

\begin{proof}
It suffices to replace $\log \Delta Q$ by $sh$ in the proof of Lemma \ref{lap_term}.
\end{proof}

We finally turn to the sum of terms $a_{\nu,j}$, given by (see \eqref{correct} for $\nu=0$),
\begin{equation}\label{anjkk}
a_{\nu,j}=\calB(r_{\nu,\tau(j)})
+ \frac{k'(r_{\nu,\tau(j)})^2 }{2} \frac{1}{d_{2}}  + \frac{k''(r_{\nu,\tau(j)})}{2}\frac{1}{d_{2}} +\frac{k'(r_{\nu,\tau(j)})}{r_{\nu,\tau(j)}}\frac{1}{d_{2}}  -\frac{k'(r_{\nu,\tau(j)})}{2} \frac{d_{3}}{d_{2}^{2}}.
\end{equation}

Here $k=sh+\alpha\ell$ and $\ell(z)=2\log|z|$; the definitions of the functions
$$\mathcal{B}(r_{\nu,\tau(j)}),\quad d_2=d_2(r_{\nu,\tau(j)}),\quad d_3=d_3(r_{\nu,\tau(j)})$$ are found in \eqref{br} and \eqref{dlnk}, respectively.

We also remind of the notation
$e_{\nu,f}$ and $v_{\nu,f}$ for the quantities in \eqref{ehnu} and \eqref{vhnu}.

\begin{lem}\label{mult_term} For $|s|\le \log n$ we have as $n\to\infty$,
\begin{align*}\frac{1}{n}\sum\limits_{j=D_n}^{m_0-1} a_{0,j}&+\frac 1 n \sum_{\nu=1}^N\sum_{j=m_{\nu-1}}^{m_\nu-1}a_{\nu,j}=F_Q[\sigma]-\frac 1 4\sum_{\nu=0}^N\log\frac {\Delta Q(b_\nu)}
{\Delta Q(a_\nu)}+\frac 1 3 \log b_N-\frac 1 {12}\log \frac {D_n} n \\
&+\frac 1 3\sum_{\nu=0}^{N-1}\log\rho_\nu+\frac 16\log\Delta Q(0)+e_k+\frac {s^2} 2 v_{0,h}+\frac 1 2 \sum_{\nu=1}^N v_{\nu,k}+\frac{(1+2\alpha)s}{2} (h(b_0)-h(0))\\
&+\frac 1 2\sum_{\nu=1}^N (k(b_\nu)-k(a_\nu)) -\frac {\alpha+\alpha^2} 2\log \frac {D_n} n+(\alpha+\alpha^2)\log b_0+\frac{\alpha+\alpha^2}2\log\Delta Q(0) + \calE(n),
\end{align*}
where $\calE(n)=\bigO((D_n^{-1}+D_n^{\frac 12}n^{-\frac 12})\cdot (1+s^2))$ and  $c>0$ is arbitrarily small but fixed.
\end{lem}

\begin{proof}  For $\nu=0$, we first observe that the Riemann sum approximation in \cite[Lemma 3.6]{BKS} gives
\begin{equation}\label{Bterm}\begin{split}
\frac{1}{n} \sum_{j=D_n}^{m_0-1} \mathcal{B}(r_{0,\tau(j)})= &F_Q[\mathbb{D}_{b_0}]+\frac{1}{3}\log b_0
- \frac{1}{12}\log \frac{D_n}{n}\\
&-\frac{1}{4}\log \frac{\Delta Q(b_0)}{\Delta Q(0)}+\frac{1}{6}\log \Delta Q(0)+\bigO(D_n^{-1}+D_n^{\frac{1}{2}}n^{-\frac{1}{2}}),\\
\end{split}
\end{equation}
where $F_Q[\mathbb{D}_{b_0}]$ is defined in \eqref{Fq_disc}.

Keeping $\nu=0$, we now estimate the sums corresponding to the other terms in \eqref{anjkk}, i.e., we shall estimate the expression
\begin{align}\label{S0}
\Sigma_0:= \frac{1}{n} \sum_{j=D_n}^{m_0-1} \frac 1 {d_2}&\bigg( \frac{k'(r_{0,\tau(j)})^2 }{2}   + \frac{k''(r_{0,\tau(j)})}{2} +\frac{k'(r_{0,\tau(j)})}{r_{0,\tau(j)}}  - \frac{k'(r_{0,\tau(j)})}{2} \frac{d_{3}}{d_{2}} \bigg).
\end{align}

Using a Riemann sum approximation, we find that $\Sigma_0$ equals to
 \begin{align}\label{S1}\int_{\tau(D_n)}^{M_{0}}  \frac 1 {4\Delta Q(r_{0,\tau})}\bigg( \frac{k'(r_{0,\tau})^2 }{2 }  +  \frac{k''(r_{0,\tau})}{2} +\frac{k'(r_{0,\tau})}{r_{0,\tau}}  - \frac{k'(r_{0,\tau})} 2 \frac { g_{\tau}^{(3)} (r_{0,\tau})}{g_{\tau}^{(2)} (r_{0,\tau})}  \bigg)\,d\tau  +\bigO\Big(\frac {1+|s|} {D_n}  \Big).
\end{align}

Changing variables as in \eqref{change_var} and writing $r_n=r_{0,\tau(D_n)}$, the integral in \eqref{S1} transforms to
\begin{align}\label{transform}
\frac 1 4 \int_{r_{n}}^{b_{0}} r \bigg(  k'(r)^2 +   k''(r) + 3  \frac {k'(r)}r -   k'(r) \frac{\partial_r \Delta Q(r)}{\Delta Q(r)}  \bigg)\,dr.
\end{align}

 The integral \eqref{transform} with $k$ replaced by $sh$ is easily evaluated using the relations
\eqref{lol7}--\eqref{lol10}, giving
\begin{align*}
 \frac 1 4 \int_{r_{n}}^{b_{0}}  r&\bigg( s^2  h'(r)^2 + s h''(r) + 3s  \frac {h'(r)}r - s h'(r) \frac{\partial_r \Delta Q(r)}{\Delta Q(r)}  \bigg)\, dr \\
& = \frac{s}{2} (h(b_0)-h(0)) + se_{0,h} + \frac {s^2} 2 v_{0,h}+\bigO(r_n\cdot (1+s^2)).
\end{align*}

Recalling the asymptotic \eqref{smalltau} for $r_{0,\tau}$ and
substituting $k=\alpha\ell$ in \eqref{transform}, we compute
\begin{align*}
 \frac 1 4 \int_{r_{n}}^{b_{0}}  r\bigg( \alpha^2 &\ell'(r)^2 + \alpha\ell''(r) + 3\alpha  \frac {\ell'(r)}{r} - \alpha \ell'(r) \frac{\partial_r \Delta Q(r)}{\Delta Q(r)}  \bigg)\,dr
 \\
 &= \alpha e_{0,\ell} + \frac{\alpha}{2} \log \frac{\Delta Q(r_{n})}{\Delta Q(0)} + \alpha(1+\alpha)\log \frac{b_{0}}{r_{n}} -\frac{\alpha}{2}\\
&=\alpha e_{0,\ell}
-\frac{\alpha(1+\alpha)}{2}\log \frac{D_{n}}{n} + \alpha(1+\alpha)\log b_{0} + \frac{\alpha(1+\alpha)}{2} \log \Delta Q(0)-\frac{\alpha}{2}
+ \bigO(r_n).
\end{align*}
The contribution in \eqref{transform} proportional to ``$\alpha s$" is equal to
\begin{align}
\frac 1 4 \int_{r_{n}}^{b_{0}} r \bigg( 4\alpha s \frac{h'(r)}{r} \bigg)\,dr = \alpha s \int_{r_{n}}^{b_{0}} h'(r)dr = \alpha s \big( h(b_{0})-h(r_{n}) \big) = \alpha s \big( h(b_{0})-h(0) \big) + \bigO(r_{n}).
\end{align}
Summing up, we find
\begin{equation}\label{sum0}\begin{split}\Sigma_0=&\frac{s}{2} (h(b_0)-h(0)) + e_{0,k} + \frac {s^2} 2 v_{0,h}  -\frac{\alpha(1+\alpha)}{2}\log \frac{D_{n}}{n} +\alpha s \big( h(b_{0})-h(0) \big) \\
&+ \alpha(1+\alpha)\log b_{0} + \frac{\alpha(1+\alpha)}{2} \log \Delta Q(0) -\frac{\alpha}{2}+\bigO(n^{-\frac 1 2}D_n^{\frac 1 2}\cdot (1+s^2)).\\
\end{split}
\end{equation}

Adding \eqref{Bterm} we obtain asymptotics for the term $\frac 1 n \sum_{j=D_n}^{m_0-1}a_{0,j}$. The other terms $\frac 1 n \sum_{j=m_{\nu-1}}^{m_\nu-1} a_{\nu,j}$, for $\nu=1,\ldots,N$ are estimated as in Lemma \ref{putt}, by substituting $sh$ for the function $k$ (which is smooth in a neighbourhood of $S^\nu=A(a_\nu,b_\nu)$ for $\nu\ge 1$). Adding these contributions, it is now straightforward to finish the proof.
\end{proof}

\begin{proof}[Proof of Lemma \ref{high_sum}] We now estimate the sum $\sum_{j=D_n}^{n-1}\log h_j$, where $D_n$ is given by \eqref{dndef}.
As before, our strategy is to begin by estimating the more explicit
\begin{equation}\label{ssum}\sum_{j=D_n}^{m_0-1}\log h_{0,j}+\sum_{\nu=1}^N\sum_{j=m_{\nu-1}}^{m_\nu-1}\log h_{\nu,j}\end{equation} by using the
approximation
\begin{align*}\log h_{\nu,j}&=\frac 1 2 \log \frac {2\pi} n- n\,g_{\tau(j)}(r_\nu,\tau(j))-\frac 1 2 \log \Delta Q(r_{\nu,\tau(j)})\\
&+(1+2\alpha)\log r_{\nu,\tau(j)}+sh(r_{\nu,\tau(j)})+\frac {a_{\nu,j}} n+\epsilon_{j,n},\end{align*}
where $\epsilon_{j,n}=\bigO(\frac {1} {n^2})$ if $j/n\ge c_0>0$ and in general
$\epsilon_{j,n}=\bigO(j^{-\frac 3 2}\log^c n)$, ($c>0$).

The next thing to observe is that
\begin{equation}\label{ssum2}\sum_{j=D_n}^{n-1}\log h_j=\sum_{j=D_n}^{m_0-1}\log h_{0,j}+\sum_{\nu=1}^N\sum_{j=m_{\nu-1}}^{m_\nu-1}\log h_{\nu,j}+\sum_{\nu=0}^{N-1}T_\nu+\bigO(n^{-99})\end{equation}
where $T_\nu$ is defined as in \eqref{tndefn} (but with $\mu_\nu=\mu_\nu(s,\alpha)$ now depending also on $\alpha$). To estimate the $T_\nu$'s, we observe that proof of Lemma \ref{theta_contribution} goes through unchanged, i.e., we have as $n\to\infty$
\begin{equation}\label{zapp}T_\nu=\log[(-\mu_\nu\rho_\nu;\rho_\nu^2)_\infty]+\log[(-\mu_\nu^{-1}\rho_\nu;\rho_\nu^2)_\infty]+\bigO(n^{-1}\log^2 n),\qquad (0\le \nu\le N-1).\end{equation}
Estimating the terms in \eqref{ssum2} using the above lemmas and \eqref{zapp}, we obtain the statement of Lemma \ref{high_sum} after simplification.
\end{proof}

\subsection{Proofs of Theorem \ref{conical_expansion} and Theorem \ref{regular_expansion}}

To the sum in Lemma \ref{high_sum} we now add the sum $\sum_{j=0}^{D_n-1}\log h_j$ found in Lemma \ref{low_sum}.

Adding also $\log(n!)$ using Stirling's approximation it is straightforward to
conclude the proof of Theorem \ref{conical_expansion}. Recalling that the Euler characteristic of a central disk droplet is
$\chi(S)=1$, we also finish our proof of the central disk part of Theorem \ref{regular_expansion} as the special case $\alpha=0$. $\qed$

\section{Shallow outposts}\label{Sec_shallow}
We now prove Theorem \ref{shallow_expansion}. We thus set $S=A(a,b)$ and $S^*=S\cup \{|z|=t\}$ where $0\le a<b<t$ and fix a suitable, radially symmetric function $h(z)=h(r)$, $r=|z|$.

As before, the partition function with respect to $\tilde{Q}=Q-\frac s n h$ is $\log Z_n=\log n!+\sum_{j=0}^{n-1}\log h_j$, where
$$
h_j=2\int_0^\infty r e^{sh(r)} e^{-ng_{\tau(j)}(r)}\, dr,\qquad g_{\tau}(r)=q(r)-2\tau\log r,\qquad \tau(j)=j/n.
$$

The set $S^*\cap[0,\infty)$ decomposes into the components $C_0=[a,b]$ and $C_1=\{t\}$.

Now consider the equation \begin{equation}\label{mutter}\frac d {dr}g_{\tau(j)}(r)=0,\end{equation}
and write $L_n=C\log n$, where $C$ is large enough.
If $j< n-L_n$, then \eqref{mutter} has exactly one significant solution $r_{0,j}$ in $C_0$, while if $j\ge n-L_n$ we must take into consideration two solutions $r_{0,j}$ near $C_0$ and $r_{1,j}$ near $C_1$. (It is straightforward to supply a proof by modification of our argument for Lemma \ref{onepk}.)

With $\eps_n$ as in \eqref{dnen}, we write
$$h_{k,j}=\int_{\{|r-r_{k,j}|<\eps_n\}}2re^{sh(r)}e^{-ng_{\tau(j)}(r)}\, dr,$$
where $k\in \{0,1\}$ if $j \geq n-L_{n}$ while $k=0$ if $j< n-L_{n}$.
By Lemma \ref{loc_lem}, we have, as $n\to\infty$, that
$h_j=(h_{0,j}+h_{1,j})\cdot (1+\bigO(n^{-100}))$ if $j \geq n-L_{n}$ while $h_j=h_{0,j} \cdot (1+\bigO(n^{-100}))$ if $j < n-L_{n}$.

To study the free energy we write
\begin{equation}\label{splitting_shallow}
\sum_{j=0}^{n-1} \log h_j = \sum_{j=0}^{n-1} \log h_{0,j} + \sum_{j=n- L_n}^{n-1} \log\Big(1+ \frac{h_{1,j}}{h_{0,j}}\Big) + \bigO(\frac{1}{n^{99}}).
\end{equation}

The first sum on the right has the same asymptotics as in Theorem \ref{regular_expansion} in the case of a single component. It remains only to study the second sum.

\begin{proof}[Proof of Theorem \ref{shallow_expansion}]
Write $\rho = b/t$ and $\mu = e^{s(h(t) -h(b) )} \sqrt{\frac{\Delta Q(b)}{\Delta Q(t)}}$. In a similar way as in  Lemma \ref{theta_contribution} we have
	$$
	\sum_{j=n-L_n}^{n-1}  \log\Big(1 + \frac{h_{1,j}}{h_{0,j}} \Big) = \sum_{j=n-L_n}^{n-1} \log(1 + \mu \rho^{2(n-j)-1}    ) + \bigO\left(\frac{1+|s|}{n}\right).
	$$

	Shifting the summation index in the sum on the right and letting $n\to\infty$, we conclude that
$
\sum_{j=n- L_n}^{n-1} \log\Big(1+ \frac{h_{1,j}}{h_{0,j}}\Big) = \log[(-\mu\rho;\rho^2)_\infty] +\bigO(n^{-1}(1+|s|))$ with uniform convergence for $|s|\le \log n$.
The proof is complete.
\end{proof}



\medskip \noindent \textbf{Acknowledgements.} CC acknowledges support from the Swedish Research Council, Grant No. 2021-04626.


\begin{thebibliography}{999}
\bibitem{AR2017} Adhikari, K., Reddy, N.K., \textit{Hole probabilities for finite and infinite Ginibre ensembles}, Int. Math. Res. Not. IMRN \textbf{2017} (2017), no. 21, 6694--6730.

\bibitem{A97} Akemann, G., \textit{Universal correlators for multi-arc complex matrix models}, Nuclear Phys. B \textbf{507} (1997), no. 1-2, 475--500.

\bibitem{ABES2023} Akemann, G., Byun, S.-S., Ebke, M., Schehr, G., \textit{Universality in the number variance and counting statistics of the real and symplectic Ginibre ensemble}, J. Phys. A \textbf{56} (2023), Paper no. 495202, 53 pp.

\bibitem{APS2009} Akemann, G., Phillips, M.J., Shifrin, L., \textit{Gap probabilities in non-Hermitian random matrix theory}, \textit{J. Math. Phys.} \textbf{50} (2009), no. 6, 063504, 32 pp.

\bibitem{A} Ameur, Y., \textit{A localization theorem for the planar Coulomb gas in an external field}, Electron. J. Probab. \textbf{26} (2021), Paper no. 46, 21 pp.

\bibitem{ACC} Ameur, Y., Charlier, C., Cronvall, J., \textit{The two-dimensional Coulomb gas: fluctuations through a spectral gap}, arXiv:2210.13959 (to appear in Arch. Rational Mech. Anal.)

\bibitem{ACCL1} Ameur, Y., Charlier, C., Cronvall, J., Lenells, J., \textit{Exponential moments for disk counting statistics
at the hard edge of random normal matrices}, J. Spectr. Theory \textbf{13} (2023), no. 3, pp. 841--902.

\bibitem{ACCL} Ameur, Y., Charlier, C., Cronvall, J., Lenells, J., \textit{Disk counting statistics near hard edges of random normal matrices: the multi-component regime}, Adv. Math. \textbf{441} (2024),
Paper No. 109549, 55 p.

\bibitem{ACM} Ameur, Y., Charlier, C., Moreillon, P., \textit{Eigenvalues of truncated unitary matrices:
disk counting statistics}, Monatsh. Math. \textbf{204} (2024), no. 2, 197--216.



\bibitem{AHM0} Ameur, Y., Hedenmalm, H., Makarov, N., \textit{Fluctuations of eigenvalues of random normal matrices},
Duke Math. J. \textbf{159} (2011), 31--81.

\bibitem{AHM} Ameur, Y., Hedenmalm, H., Makarov, N., \textit{Random normal matrices and Ward identities}, Ann. Probab. \textbf{43} (2015), 1157--1201.

\bibitem{AKS} Ameur, Y., Kang, N.-G., Seo, S.-M., \textit{On boundary confinements for the Coulomb gas}, Anal. Math. Phys. \textbf{10}, paper no. 68 (2020).

\bibitem{AKS2} Ameur, Y., Kang, N.-G., Seo, S.-M., \textit{The random normal matrix model: insertion of a point charge}, Potential Anal. \textbf{58} (2023), 331-372.

\bibitem{AAR} Andrews, G., Askey, R., Roy, R.  \textit{Special Functions}. Encyclopedia of Mathematics and its Applications, Cambridge University Press 1999.

\bibitem{BBNY} Bauerschmidt, R., Bourgade, P., Nikula, M., Yau, H.-T., \textit{The two-dimensional Coulomb plasma: quasi-free
approximation and central limit theorem.} Adv. Theor. Math. Phys. \textbf{23}, 841-1002 (2019).

\bibitem{Berezin} Berezin, S., \textit{Functional limit theorems for constrained Mittag-Leffler ensemble in hard edge scaling}, arXiv:2308.12658.


\bibitem{Berman} Berman, R., \textit{K\"{a}hler-Einstein metrics emerging from free fermions and statistical mechanics}, J. High Energ. Phys. \textbf{2011}, 106 (2011).

\bibitem{BL} Bertola, M., Lee, S.-Y., \textit{First colonization of a spectral outpost in random matrix theory}, Constr. Approx. \textbf{30} (2009), 225-263.

\bibitem{Bill} Billingsley,  P., \textit{Probability and measure.} Anniversary edition. Wiley Series in Probability and Mathematical Statistics. John Wiley \& Sons, Inc., New York, 2012.

\bibitem{BG} Borot, G., Guionnet, A., \textit{Asymptotic expansion of $\beta$ matrix models in the multi-cut regime}, Forum Math. Sigma \textbf{12} (2024), Paper No. e13, 93 pp.


\bibitem{BC FH 2022} Byun, S.-S., Charlier, C., \textit{On the characteristic polynomial of the eigenvalue moduli of random normal matrices},  Constr. Approx. (2024) \texttt{https://doi.org/10.1007/s00365-024-09689-x.}


\bibitem{BKS} Byun, S.-S., Kang, N.-G., Seo, S.-M., \textit{Partition functions of determinantal and Pfaffian Coulomb gases with radially symmetric potentials}, Commun. Math. Phys. \textbf{401} (2023), 1627-1663.

\bibitem{BF} Byun, S.-S., Forrester, P.J., \textit{Progress on the study of the Ginibre ensembles}  KIAS Springer Series in Mathematics, Vol. 3, 2025.


\bibitem{ByunPark2024} Byun, S.-S., Park, S., \textit{Large gap probabilities of complex and symplectic spherical ensembles with point charges}, arXiv:2405.00386.

\bibitem{ByunSeoYang} Byun, S.-S., Seo, S.-M., Yang, M., \textit{Free energy expansions of a conditional GinUE and large deviations of the smallest eigenvalue of the LUE}, arXiv:2402.18983.

\bibitem{CFTW} Can, T., Forrester, P.J., T\'{e}llez, G., Wiegmann, P., \textit{Exact and asymptotic features of the edge density profile for the one component plasma in two dimensions}, J. Stat. Phys. \textbf{158} (2015), no.5, 1147--1180.

\bibitem{C2021 FH} Charlier, C., \textit{Asymptotics of determinants with a rotation-invariant weight and discontinuities along circles}, Adv. Math. \textbf{408} (2022), Paper No. 108600, 36 p.

\bibitem{C} Charlier, C., \textit{Large gap asymptotics on annuli in the random normal matrix model}, Math. Ann. \textbf{388} (2024), 3529--3587.

\bibitem{C2} Charlier, C., \textit{Hole probabilities and balayage of measures for planar Coulomb gases}, arXiv:2311.15285.

\bibitem{CG FH 2021} Charlier, C., Gharakhloo, R., \textit{Asymptotics of Hankel determinants with a Laguerre-type or Jacobi-type potential and Fisher-Hartwig singularities}, Adv. Math. \textbf{383} (2021), Paper No. 107672, 69 pp.

\bibitem{CFWW} Charlier, C., Fahs, B., Webb, C., Wong, M.-D., \textit{Asymptotics of Hankel determinants with a multi-cut regular potential and Fisher-Hartwig singularities}, arXiv:2111.08395. (to appear in Mem. Amer. Math. Soc.).


\bibitem{Cl} Claeys, T., \textit{The birth of a cut in unitary random matrix ensembles}, Int. Math. Res. Not. IMRN \textbf{2008} (2008), no.6, Art. ID rnm166, 40 pp.

\bibitem{ClaeysGravaMcLaughlin} Claeys, T., Grava T., McLaughlin, K.T.-R., \textit{Asymptotics for the partition function in two-cut random matrix models}, Comm. Math. Phys. \textbf{339} (2015), no. 2, 513--587.



\bibitem{Cr} Cronvall, J., To appear.

\bibitem{DS} Dea\~{n}o, A., Simm, N.J., \textit{Characteristic polynomials of complex random matrices and Painlev\'{e} transcendents}, Int. Math. Res. Not. IMRN \textbf{2022} (2022), no.1, 210--264.

\bibitem{Deift} Deift, P., \textit{Orthogonal polynomials and random matrices: a Riemann-Hilbert approach}, Courant Lecture Notes in Mathematics, 3.

\bibitem{DIZ1997} Deift, P., Its A., Zhou, X., \textit{A Riemann-Hilbert approach to asymptotic problems arising in the theory of random matrix models, and also in the theory of integrable statistical mechanics}, Ann. of Math. (2) \textbf{146} (1997), no. 1,
    149--235.

\bibitem{Eynard} Eynard, B., \textit{Universal distribution of random matrix eigenvalues near the ``birth of a cut'' transition}, J. Stat. Mech. Theory Exp. (2006), no.7, P07005, 33 pp.

\bibitem{FK} Faddeev, L.D., Kashaev, R.M., \textit{Quantum dilogarithm}, Modern Phys. Lett. A \textbf{9} (1994), no. 5, 427--434.

\bibitem{FSW} Fefferman, C., Shapiro, J., Weinstein, M. I.,
\textit{Lower bound on quantum tunneling for strong magnetic fields}, SIAM J. Math. Anal. \textbf{54} (2022), no.1, 1105-1130.

\bibitem{FahsKraI} Fahs, B., Krasovsky, I., \textit{Splitting of a gap in the bulk of the spectrum of random matrices}, Duke Math. J. \textbf{168} (2019), no. 18, 3529--3590.

\bibitem{FahsKraII} Fahs, B., Krasovsky, I., \textit{Sine-kernel determinant on two large intervals}, Comm. Pure Appl. Math. (2023), \texttt{https://doi.org/10.1002/cpa.22147}.

\bibitem{ForresterHoleProba} Forrester, P.J.,  \textit{Some statistical properties of the eigenvalues of complex random matrices}, Phys. Lett. A \textbf{169} (1992), no. 1-2, 21--24.

\bibitem{F} Forrester, P.J., \textit{A review of exact results for fluctuation formulas in random matrix theory}, Probab. Surv. \textbf{20} (2023), 170--225.

\bibitem{Fo} Forrester P.J., \textit{Log-gases and Random Matrices} (LMS-34), Princeton University Press, Princeton 2010.

\bibitem{GR} Gasper, G., Rahman, M., \textit{Basic hypergeometric series}, Cambridge 1990.

\bibitem{HZ} Hassell, A., Zelditch, S., \textit{Determinants of Laplacians in Exterior Domains}, Internat. Math. Res. Notices \textbf{1999} (1999), no.18, 971--1004.

\bibitem{HK} Helffer, B., Kachmar, A., \textit{Quantum tunneling in deep potential wells and strong magnetic fields revisited}, Pure Appl. Anal. \textbf{6} (2024), no. 2, 319--352.


\bibitem{JMP} Jancovici, B., Manificat, G., Pisani, C., \textit{Coulomb systems seen as critical systems: Finite-size effects in two dimensions.} J. Statist. Phys. \textbf{76} (1994), 307--329.

\bibitem{JV} Johansson, K., Viklund, F., \textit{Coulomb gas and Grunsky operator on a Jordan domain with corners}, arXiv:2309.00308.

\bibitem{Ka} Kalvin, V.,
\textit{Polyakov-Alvarez type comparison formulas for determinants of Laplacians on Riemann surfaces with conical singularities}, J. Funct. Anal. \textbf{280} (2021), no.7, Paper No. 108866, 44 pp.

\bibitem{Ke2} Kemp, A.W., \textit{Characterizations of a discrete normal distribution}, Journal of Statistical Planning and Inference, 63 (1997), 223–229.

\bibitem{Ke} Kemp, A.W., \textit{Heine-Euler extensions of the poisson distribution}, Communications in Statistics - Theory and Methods
\textbf{21} (1992), 571--588.

\bibitem{K} Khinchin, A.I., \textit{Mathematical foundations of information theory}, Dover 1957.


\bibitem{Kl} Klevtsov, S.,
\textit{Lowest Landau level on a cone and zeta determinants}, J. Phys. A \textbf{50} (2017), no.23, 234003, 16 pp.

\bibitem{LR} Lee, S.-Y., Riser, R., \textit{To appear}.


\bibitem{MH} Mathai, A.M., Haubold, A.J., \textit{Special Functions for Applied Scientists}, Springer 2008.

\bibitem{M} Mehta, M.L., \textit{Random matrices}, Third Edition, Academic Press 2004.

\bibitem{Mo} Mo, M.Y., \textit{The Riemann-Hilbert approach to double scaling limit of random matrix eigenvalues near the ``birth of a cut'' transition}, Int. Math. Res. Not. IMRN \textbf{2008} (2008), no. 13, Art. ID rnn042, 51 pp.

\bibitem{NIST} Olver, F.W.J., Olde Daalhuis, A.B., Lozier, D.W., Schneider, B.I., Boisvert, R.F., Clark, C.W., Miller, B.R., Saunders, B.V., \textit{NIST Digital Library of Mathematical Functions}, http://dlmf.nist.gov/, V. 1.0.13, 2016-09-16.


\bibitem{Seo2021} Seo, S.-M., \textit{Edge behavior of two-dimensional Coulomb gases near a hard wall}, Ann. Henri Poincar\'{e} \textbf{23} (2021), 2247--2275.

\bibitem{Se} Serfaty, S., \textit{Gaussian fluctuations and free energy expansion for Coulomb gases at any temperature},
Ann. Inst. Henri Poincar\'{e} Probab. Stat. \textbf{59} (2023), 1074-1142. 

\bibitem{ST} Saff, E.B., Totik, V., \textit{Logarithmic potentials with
external fields}, Springer 1997.

\bibitem{S} Simon, B., \textit{Advanced Complex Analysis: A comprehensive course in analysis 2B}, American Mathematical Society 2015.


\bibitem{W} Wang, Y., \textit{Equivalent descriptions of the Loewner energy}, Invent. Math. \textbf{218}, 573-621 (2019).


\bibitem{WW} Webb, C., Wong, M.D., \textit{On the moments of the characteristic polynomial of a Ginibre random matrix}, Proc. Lond. Math. Soc. (3) \textbf{118} (2019), no. 5, 1017--1056.


\bibitem{ZW} Zabrodin, A., Wiegmann, P., \textit{Large $N$ expansion for the 2D Dyson gas}, J. Phys. A \textbf{39} (2006), no.28, 8933--8964.
\end{thebibliography}
\end{document}